\documentclass[a4paper,11pt]{amsart}
\usepackage{amssymb,amsfonts,amsxtra,
	mathrsfs,graphicx,verbatim,stmaryrd,hyperref,cite,color, tikz-cd, mathabx}
\usepackage[margin=8mm,
marginparwidth=20mm,     
marginparsep=1mm,       
]
{geometry}
\usepackage[all]{xy}
\xyoption{line}
\usepackage{fullpage}
\usepackage{euscript}
\usepackage{xcolor}

\usepackage{hyperref}

\newtheorem{theorem}{Theorem}[section]
\newtheorem{cor}[theorem]{Corollary}
\newtheorem{lem}[theorem]{Lemma}
\newtheorem*{hyp}{Hypothesis}
\newtheorem{con}[theorem]{Conjecture}
\newtheorem{prop}[theorem]{Proposition}

\theoremstyle{definition}

\newtheorem{example}[theorem]{Example}
\newtheorem{defi}[theorem]{Definition}
\newtheorem{rem}[theorem]{Remark}
\newcommand{\tw}{\mathbf{tw}}
\renewcommand{\thefootnote}{\fnsymbol{footnote}}

\numberwithin{equation}{section}
\DeclareMathOperator{\colim}{colim}
\DeclareMathOperator{\g}{\mathfrak g}

\DeclareMathOperator{\C}{\mathcal C}
\DeclareMathOperator{\CE}{CE}
\DeclareMathOperator{\B}{B}
\DeclareMathOperator{\U}{U}
\DeclareMathOperator{\HH}{H}
\DeclareMathOperator{\Harr}{Harr}

\DeclareMathOperator{\DGA}{DGA}
\DeclareMathOperator{\DGC}{DGC}
\DeclareMathOperator{\DGCC}{DGCC}
\DeclareMathOperator{\DGL}{DGL}
\DeclareMathOperator{\op}{op}
\DeclareMathOperator{\Hom}{Hom}
\DeclareMathOperator{\End}{End}
\DeclareMathOperator{\HCE}{HCE}

\DeclareMathOperator{\co}{co}
\DeclareMathOperator{\IM}{Im}

\def\MMod{\operatorname{\!-Mod}\nolimits}

\def\CComod{\operatorname{\!-Comod}\nolimits}

\DeclareMathOperator*{\DGVect}{DGVect}
\DeclareMathOperator*{\res}{res}
\DeclareMathOperator*{\DGVectpc}{DGVect_{pc}}
\DeclareMathOperator*{\DGVectop}{DGVect^{op}}
\DeclareMathOperator*{\RHom}{\operatorname{RHom}}

\DeclareMathOperator*{\Perf}{\operatorname{Perf}}
\DeclareMathOperator*{\holim}{\operatorname{holim}}
\DeclareMathOperator*{\hocolim}{\operatorname{hocolim}}
\DeclareMathOperator{\MC}{MC}
\DeclareMathOperator{\str}{str}
\DeclareMathOperator{\Str}{tr}
\newcommand{\pr}{\mathbf{pr}}
\newcommand{\ad}{\mathbf{ad}}
\newcommand {\emptycomment}[1]{}

\def\Z{\mathbb{Z}}

\def\id{\operatorname{id}}

\newcommand{\bk}{{\mathbf{k}}}
\newcommand{\bd}{{\mathbf{d}}}
\newcommand{\huaL}{\mathbb{L}}

\newcommand{\huaS}{\mathbf{S}}
\newcommand{\huaT}{\mathbf{T}}
\newcommand{\huaTT}{\mathcal{T}}
\newcommand{\huathick}{\mathbf{thick}}
\newcommand{\huaB}{\mathcal{B}}
\newcommand{\huaC}{\mathcal{C}}

\thanks{}

\begin{document}
\begin{abstract}A well-known and old result of Hazewinkel and Koszul states that the cohomology of a finite-dimensional Lie algebra is isomorphic, up to a suitable shift, to its twisted homology, a Lie-theoretical version of Poincar\'e duality. This paper establishes an analogue of this result for \emph{graded} (or super) Lie algebras and, more generally, differential graded Lie algebras.
 Closely related to this result is a calculation of the cohomology of a graded Lie algebra $\g$ with coefficients in its universal enveloping algebra $\U(\g)$ as a one-sided module. This cohomology turns out to be one-dimensional and serves as a dualizing module for the cohomology of $\g$. Moreover, it is shown that for a unimodular graded Lie algebra $\g$, the derived category of $\U(\g)$ has a Calabi-Yau structure. As a consequence, the category of rational infinity local systems on a simply connected topological space with totally finite-dimensional rational homotopy groups, has a Calabi-Yau structure, a generalization of Poincar\'e  duality for elliptic spaces.

\end{abstract}
\title[Unimodular [Duality for graded Lie algebras]{Duality for graded Lie algebras}
\author{Andrey Lazarev}
	\address{Department of Mathematics and Statistics, Lancaster University, Lancaster LA1 4YF, UK}
	\email{a.lazarev@lancaster.ac.uk}
		
	\author{Rong Tang}
	\address{Department of Mathematics, Jilin University, Changchun 130012, Jilin, China, and Department of Mathematics and Statistics, Lancaster University, Lancaster LA1 4YF, UK}
	\email{r.tang8@lancaster.ac.uk}
\renewcommand{\thefootnote}{}
\maketitle
\tableofcontents
\section{Introduction}
Given an augmented differential graded (dg) algebra $A$ over a field $\bk$ and its bar-construction coalgebra $\B(A)$, the categories of dg $A$-modules and of $\B(A)$-comodules are equivalent as $\infty$-categories (or, more precisely, there is a Quillen equivalence between them as model categories). This is one of the manifestations of dg Koszul duality \cite{Pos} and in this paper we study certain duality functors in these $\infty$-categories. On the algebra side, it is the functor $$d:X\mapsto d(X):={\RHom}_{A}(X,A)$$ and on the coalgebra side it is a functor that is simply the linear dual on the subcategory of \emph{finite-dimensional} dg comodules but is rather more involved for general comodules. It was shown in \cite{BCL} that these duality functors correspond to each other across Koszul duality. The right $A$-module $\RHom_{A}(\bk,A)$ is called the (one-sided) dualizing complex for an augmented dg algebra $A$. Following \cite{BCL}, we call  $A$ a \emph{Gorenstein} dg algebra if its dualizing complex is quasi-isomorphic to the one-dimensional $\bk$-vector space with a trivial right $A$-action (i.e. given by the augmentation map).

We will consider the situation when $A$ is a dg Hopf algebra ({such as the dg algebra of singular chains on based loops on a topological space}). Our first result (Theorem \ref{thm:Serre}) states that the derived category $D(A)$ of a dg Hopf algebra $A$ has an open Serre duality  given by a pair
		$(\Perf(A),\operatorname{pvd}(A))$ of subcategories of $D(A)$ and the open Serre functor given by tensoring with the dualizing complex. 	Here $\Perf(A)$ consists of perfect (compact) objects in $D(A)$ and $\operatorname{pvd}(A)$ is formed by $A$-modules with finite-dimensional cohomology. If $A$ is Gorenstein, then $D(A)$ has a further refinement, called the (open) Calabi-Yau (CY) structure. That means, roughly, that if  $N\in \operatorname{pvd}(A)$ and $M\in \Perf(A)$, then $\Hom_{D(A)}(N,M)$ and $\Hom_{D(A)}(M,N)$ are $\bk$-dual to each other up to a shift. Note that this notions of an open Serre duality and
 an open CY structure are generalizations of the usual ones (when the two subcategories coincide). This was explicitly spelled out in \cite{BCL} but was known for a long time, at least since Keller's paper \cite{Kel}, cf. especially Lemma 3.4. of op.cit. The notion of a triangulated CY category stems from  derived categories of CY varieties and Kontsevich's homological mirror symmetry, cf. \cite{Kon}.

We next specialize to the case $A=\U(\g)$, the universal enveloping algebra of a graded Lie algebra $\g$ that we assume to be finite-dimensional. We show Theorem \ref{thm:super-g-trace} that the dualizing complex for $\U(\g)$ has one-dimensional cohomology and the action of $\g$ on it is through the trace of the adjoint representation. This is the main result of this paper and, though its formulation is simple, the proof is surprisingly involved. Namely, one-dimensionality follows from a relatively straightforward argument with spectral sequences but the calculation of the $\g$-action requires an application of homological perturbation theory and minimal models for $A_\infty$-modules. This result can also be viewed as a $\g$-equivariant analogue of the well-known homological interpretation of the Berezinian (or superdeterminant).

We also have a similar result for a \emph{connective} dg Lie algebra (such as the Lie-Quillen rational model of a topological space).

A consequence of this is that the universal enveloping algebra of a graded unimodular Lie algebra is Gorenstein and so, its derived category is CY with respect to subcategories of perfect and proper dg $\g$-modules. Applied to the Lie-Quillen model of a simply-connected topological space $A$ with totally finite-dimensional rational homotopy groups, we obtain that the category of infinity local systems on $X$ (or the derived category of cohomologically locally constant complexes of sheaves on $X$) have a CY property, a generalization of Poincar\'e duality for elliptic spaces \cite{FHT}.

This also generalizes the fact that the universal enveloping algebra $\U(\g)$ of ordinary (ungraded) unimodular Lie algebra $\g$ is Gorenstein and CY cf. \cite{BCL, HR}. One essential difference between the graded and ungraded cases is that $\U(\g)$ is homologically smooth for an ungraded $\g$ whereas this is not true for graded Lie algebras (e.g. for a one-dimensional graded Lie algebra $\g$ concentrated in an odd degree, $\U(\g)$ is an exterior algebra on one generator, and it has an infinite global dimension).

These results admit an appropriate translation to the Koszul dual side, for dg comodules over the bar-construction $\B(\U(\g))$. In fact, it is more convenient to pass to the CE coalgebra $\CE_\bullet(\g)$ which is weakly equivalent to $\B(\U(\g))$ (and thus, has an equivalent coderived category) but is much smaller and easier to work with. We show (Theorem \ref{derived-CE-comod}) that $\CE_\bullet(\g)$ is derived dual as a comodule over itself, to a certain twist of itself (up to a shift) and, in the case of a unimodular Lie algebra $\g$, this twist is trivial. This is a far-reaching generalization of old results of Hazewinkel \cite{Haz} and Koszul \cite{Kos}. In this connection, we mention that there is another generalization of Poincar\'e duality for ordinary Lie algebras where the latter are replaced by Lie algebroids \cite{Wein}, or, still more generally, Lie-Reinhart pairs \cite{Hue} and $Q$-manifolds \cite{Vor}. Investigating the relationship between these results and ours is an interesting open problem.

Next, we study a potential generalization of the main result described above, to the case of dg Lie algebras. In this situation, it is more convenient to switch to a more flexible language of $L_\infty$-algebras (or strong homotopy Lie algebras). An $L_\infty$-algebra $(\g,\ell)$
is most concisely defined through its representing dg algebra $\CE^\bullet(\g)$ which is a formal power series algebra  $\hat{\huaS}(\g^*[-1])$ on the shifted dual space to $\g$ supplied with a continuous derivation $\ell$ having no constant term. In the case of a  graded Lie algebra $\g$, $\CE^\bullet(\g)$ is the ordinary Chevalley-Eilenberg (CE) algebra of $\g$ and $\ell$ has only the quadratic term. The $\bk$-linear continuous dual to $\CE^\bullet(\g)$ is denoted by $\CE_\bullet(\g)$ and agrees with the CE homological complex of $\g$ when $\g$ is a graded Lie algebra. In this case, $\CE_\bullet(\g)$ is a conilpotent cocommutative dg coalgebra. Strong homotopy Lie algebras arise in a great variety of contexts such as deformation quantization \cite{Kon'}, string field theory \cite{Zwi} and others. Their representing dg algebras can be viewed as linear duals to fibrant objects in the model category of cocommutative conilpotent dg coalgebras, cf. \cite{Hin, Pri} for this point of view.

Given an $L_\infty$-algebra $(\g,\ell)$, the divergence $\nabla(\ell)$ of the derivation $\ell$ is a CE 1-cocycle  which, in the case of a graded Lie algebra, corresponds to the trace of the adjoint representation of $\g$.  Thus, it is natural to conjecture that the derived dual to $\CE_\bullet(\g)$ as a comodule over itself, is, up to an appropriate shift, its twist by the cocycle $\nabla(\ell)$. Note that Theorem \ref{derived-CE-comod} is a special case of this conjecture when $\g$ is a graded Lie algebra.

Working on the Koszul dual side of dg $\CE_\bullet(\g)$-comodules, we develop appropriate tools (having independent interest) for proving this conjecture and obtain some partial results. These complement the results obtained on the side of $\g$-modules described above in the sense that neither of them are special cases of the other but both are special cases of above conjecture.

The paper is organized as follows. Section \ref{sec:recollection} contains background on dg Koszul duality between modules over a dg Lie algebra and comodules over its CE coalgebra . In Section \ref{section:Hopf} we study the derived duality functor in the derived category $D(A)$ of a dg Hopf algebra $A$ and, in particular, prove that if $A$ is Gorenstein, then $D(A)$ is CY (Corollary \ref{cor:HopfCY}). In Section \ref{sec:dualizing} it is proved that the dualizing complex for a proper dg Lie algebra has one-dimensional cohomology. Section \ref{sec:technical}, which is the technical heart of the present paper, is dedicated to proving that the action of finite-dimensional graded Lie algebra on its dualizing complex is via the (super) trace of its adjoint representation. This section also contains topological applications mentioned above. In Section \ref{sec: Laurent}  integrals on rings of formal Laurent series are discussed and a formula for the adjoint of a derivation is obtained (Proposition \ref{prop:adjoint}); this formula must be a standard result in differential geometry, at least in the ungraded context but, since we were unable to find a reference, a complete proof was given. Section \ref{sec:strange assumption} contains results on the structure of the derived dual to $\CE_\bullet(\g)$  where $\g$ is an $L_\infty$-algebra satisfying certain (unfortunately, quite strong) assumptions; these are Propositions \ref{prop:CEdual} and \ref{prop:CEdual'}. Of some independent interest here  is a construction of the derived pseudocompact completion. The last section contains calculations with $L_\infty$-algebras of small dimensions. In Appendix \eqref{A}, we prove an equivalence of two definitions of a graded $n$-CY category.

\subsection{Notation and conventions}
Throughout this paper, we will denote by $\bk$ a fixed field of characteristic zero. Unmarked tensor products and Homs will be understood as taken over $\bk$. We will mostly be working in the category $\DGVect$ of differential $\Z$-graded (dg) vector spaces over $\bk$; our grading will always be cohomological and for a dg vector space $V$, we write $\HH^\bullet(V)$ for its cohomology; the degree of a homogeneous element $v\in V$ is denoted by $|v|$. The shift $V[k]$ of a graded vector space $V$ is defined as $V[k]^i=V^{i+k}$ and an element of $V[k]$ is denoted by $v[k],v\in V$. An (associative, unital) dg algebra is a monoid in $\DGVect$ and a (coassociative, counital) dg-coalgebra is a comonoid in $\DGVect$. The category of dg-algebras will be denoted by $\DGA$. A graded dg algebra $A$ is \emph{augmented} if the unit map $\bk\to A$ admits a retract $\epsilon:A\to\bk$, the augmentation.  Maps between augmented algebras are supposed to preserve augmentation. The augmentation ideal $\bar{A}$ is the kernel of $\epsilon$. The category of augmented dg-algebras will be denoted by $\DGA_*$. We will also work with the category $\DGL$ of dg Lie algebras. Given a dg Lie algebra $\g$, we denote by $\U(\g)$ its universal enveloping dg algebra.

Given two dg modules $M, N$ over a dg algebra $A$ (assumed to be left unless indicated otherwise), we write $\Hom_A(M,N)$ for the dg vector space of $A$-homomorphisms from $M$ to $N$ and $\RHom_A(M,N)$ for the corresponding derived functor (which is, therefore, an object of $D(\bk)$).

A \emph{pseudocompact} (pc) dg vector space is a projective limit of finite dimensional vector spaces and the category of pc dg vector  spaces with continuous chain maps will be denoted by $\DGVectpc$. The functor of linear duality $V\mapsto V^*$ establishes an anti-equivalence from $\DGVect$ to $\DGVectpc$ with the inverse functor being the functor of continuous duality (which we will denote by the same symbol $W\mapsto W^*$ for a pc dg vector space $W$). Therefore, $\DGVectop$, the opposite category to $\DGVect$ is equivalent to $\DGVectpc$ and we will frequently not distinguish between these two categories. In particular, the data of a (coassociative, counital) dg coalgebra are equivalent under this correspondence to that of a monoid in $\DGVectpc$, i.e. a pc dg algebra. A dg coalgebra is \emph{conilpotent} if its dual pc dg algebra is $\bk$-augmented and local, i.e. its augmentation ideal is a unique two-sided dg maximal ideal. The category of conilpotent dg coalgebras will be denoted by $\DGC_*$. Similarly, the category of conilpotent cocommutative dg coalgebras will be denoted by $\DGCC$.

A \emph{proper} dg vector space is one with totally finite-dimensional cohomology. Similarly, a dg algebra or a dg Lie algebra is proper if its cohomology is (totally) finite dimensional.

Given a dg (pc or not) algebra $A$, a Maurer-Cartan (MC) element in $A$ is an element $\xi\in A^1$ such that $d\xi+\xi^2=0$. The set of MC elements of $A$ is denoted by $\MC(A)$. Given a dg $A$-bimodule $(M,d)$ and an MC element $\xi\in\MC(A)$, define the \emph{right twist} $M^{[\xi]}$ of $M$ to be the \emph{left} dg $A$-module with underlying graded vector space $M$ supplied with the given left action of $A$ and the twisted differential $d^{[\xi]}$ given by by $d^{[\xi]}(m)=d(m)-(-1)^{|m|} m\xi$ for $m\in M$.

Similarly, a \emph{left twist} $~^{[\xi]}M$ of $(M,d)$ is the \emph{right} dg $A$-module with the same underlying graded vector space and left action as $M$ and supplied with the differential  $~^{[\xi]}d(m)=d(m)+\xi m$ for $m\in A$. This construction is most commonly applied to the case when $M=A$ or when $M$ is a free module over $A$ of the form $A\otimes V$ for a dg (possibly pc) vector space $V$. See \cite[Section 6.1]{BL-0} for more details about   right and left  twist of  dg modules.

\section{Recollection: associative and Lie-commutative Koszul duality}\label{sec:recollection}
In this section we recall Koszul duality between associative algebras and coalgebras \cite{Pos}, between Lie algebras and cocommutative coalgebras \cite{Hin} and a relationship between the two.

Given an augmented dg algebra $A$, we denote by $\B(A)$ its bar-construction, which is a cofree conilpotent coalgebra $\huaT(\bar{A}[1])$ with the differential $d_1+d_2$ where $d_1$ is induced by the internal differential in $A$ and $d_2$ is  induced by the degree $1$ linear map $\bar{A}[1]\otimes \bar{A}[1]\to \bar{A}[1]$ which is, in turn, induced by the multiplication $A\otimes A\to A$. An equivalent datum is the pc dg algebra $\hat{\huaT}(\bar{A}^*[-1])$ with a quadratic differential $\bar{A^*}[-1]\to\bar{A^*}[-1]\otimes \bar{A^*}[-1]$ called the dual bar-construction of $A$.

Similarly, given a conilpotent dg coalgebra $C$, we denote by $\Omega (C)$ its cobar-construction, which is a free algebra $\huaT(\bar{C}[-1])$ with the differential $d_1+d_2$ where $d_1$ is induced by the internal differential in $C$ and $d_2$ is induced by the degree $1$ linear map $\bar{C}[-1]\to\bar{C}[-1]\otimes\bar{C}[-1]$ which is, in turn, induced by the comultiplication $C\to C\otimes C$. The functor $\Omega$ is left adjoint to the functor $\B$.

The category $\DGA_*$ has a model structure with weak equivalences being quasi-isomorphisms and fibrations being surjective maps. The category $\DGC_*$ has a model structure with weak equivalences being maps of dg coalgebras that become quasi-isomorphisms upon applying the functor $\Omega$ and cofibrations being injective maps.

The fundamental result is that these two model categories are Quillen equivalent:
\begin{theorem}
	The pair of adjoint functors $(\Omega, \B)$ establishes a Quillen equivalence between $\DGC_*$ and $\DGA_*$.
\end{theorem}
\begin{proof}
	The corresponding statement is proved in \cite[Section 9.3]{Pos} for unital dg algebras and curved conilpotent coalgebras; see also  an earlier work \cite[Th\'eor\`eme 1.3.1.2]{Lef}. The claimed statement is then straightforward; indeed, $\DGA_*$ is just the over category of the initial object in $\DGA$ and similarly $\DGC_*$ is (isomorphic to) the over category of the initial object in curved conilpotent coalgebras.
\end{proof}

A similar result holds when an augmented dg algebra $A$ is replaced by a dg Lie algebra and a conilpotent dg coalgebra $C$ by a \emph{cocommutative} conilpotent dg coalgebra.

Namely, given a dg algebra $\g$, denote by $\CE_\bullet(\g)$
its homological CE complex,
which is a cocommutative cofree conilpotent
 coalgebra $\huaS(\g[1])$ with the differential
 $d_1+d_2$ where $d_1$ is induced by the internal differential
 in $\g$ and $d_2$ is  induced by the degree $1$ linear map ${\g}[1]\otimes {\g}[1]\to {\g}[1]$ which is, in turn, induced by the Lie bracket $\g\otimes \g\to \g$. An equivalent datum is the pc dg algebra $\hat{\huaS}({\g}^*[-1])$ with a quadratic differential ${\g^*}[-1]\to {\g^*}[-1]\otimes {\g^*}[-1]$ called the cohomological CE complex of $\g$.

 Similarly, given a conilpotent cocommutative coalgebra $D$, its Harrison complex $\Harr^\bullet(C)$  is the free Lie algebra $\mathbf{L}(\bar{D}[-1])$
 with the differential $d_1+d_2$ where $d_1$ is induced by the internal differential in $D$ and $d_2$ is induced by the degree $1$ linear map $\bar{D}[-1]\to\bar{D}[-1]\otimes\bar{D}[-1]$ which is, in turn, induced by the comultiplication $D\to D\otimes D$. Note that $\Omega(D)$ is a dg Hopf algebra and $\Harr^\bullet(D)$ coincides with the complex of primitives in $\Omega(D)$.  The functor $\Harr^\bullet$ is left adjoint to the functor $\CE_\bullet$.

 The category $\DGL$ has a model structure where weak equivalences are quasi-isomorphisms and fibrations are surjections while $\DGCC$ has a model structure with weak equivalences and fibrations being transferred from $\DGC_*$. Then the following result holds.
.

 \begin{theorem}
 	The pair of adjoint functors $(\Harr^\bullet,\CE_\bullet)$ determine a Quillen equivalence between $\DGCC$ and $\DGL$.
 \end{theorem}
 \begin{proof}
 	This is the main result of \cite{Hin}.
 \end{proof}
So,  associated to a dg Lie algebra $\g$ are two conilpotent  dg coalgebras: $\CE_\bullet(\g)$ and $\B(\U(\g))$, the first of which is cocommutative while the second is not. Then we have the following result.
\begin{prop}
	The dg coalgebras $\CE_\bullet(\g)$ and $\B(\U(\g))$ are weakly equivalent.
\end{prop}
\begin{proof}
	First recall that in order to prove that two conilpotent dg coalgebras are weakly equivalent, it suffices to construct a filtered quasi-isomorphism between them corresponding to some admissible filtrations \cite[Section 6.10]{Pos}. We will now construct such a filtered quasi-isomorphism.
	
	Note that there is a map of dg coalgebras $f:\CE_\bullet(\g)\to\B(\U(\g))$ corresponding to the canonical MC element in the convolution dg algebra $\Hom(\CE_\bullet(\g),\U(\g))$. Now introduce an increasing exhaustive filtration  on $\CE_\bullet(\g)$ by stipulating that any monomial $g_1...g_i...g_k,k\le n$ with $g_i\in\g$ belongs to the filtration component $F_n$. Similarly introduce a filtration on $\B(\U(\g))$ as induced from the PBW filtration on $\U(\g)$. It is clear that the map $f$ is compatible with the filtration just described and thus, induces
	a map on the associated graded modules \begin{equation}\label{assgr}\operatorname{gr}\CE_\bullet(\g)\to\operatorname{gr}\B(\U(\g)).
	\end{equation}
	Note that $\operatorname{gr}\B(\U(\g))$ is a complex computing $\operatorname{Ext}_{\huaS(\g)}(\bk,\bk)$ via the bar-resolution of $\bk$ over the symmetric
	 algebra $\huaS(\g)$ whereas $\operatorname{gr}\CE_\bullet(\g)$ (with the zero differential) is the Koszul complex computing the same derived functor. Thus, the (\ref{assgr}) is a quasi-isomorphism as required.
\end{proof}
Recall that associated to any dg algebra is the category $A\MMod$ of (left) dg $A$-modules which is a model category with surjective maps for fibrations and quasi-isomorphisms as weak equivalences; the corresponding homotopy category is the derived category $D(A)$ of the dg algebra $A$ (this is a standard way of constructing $D(A)$). This is the \emph{projective} model structure, all objects in it are fibrant.

Similarly,  associated to any dg coalgebra $C$ is the category $C\CComod$ of (left) dg $C$-comodules
 which is a model category in its own right, with injective maps as cofibrations and coacyclic maps as weak equivalences, \cite[Section 8.2]{Pos}; the corresponding homotopy category is the  coderived category $D^{\co}(C)$ of the dg coalgebra $C$. Two weakly equivalent dg coalgebras give rise to Quillen equivalent categories of dg comodules. Moreover, a module-comodule version of Koszul duality gives a Quillen equivalence between $A\MMod$ and $\B(A)\CComod$, cf, \cite[Section 8.4]{Pos} for any dg algebra $A$. Therefore, taking $A=\U(\g)$, the universal enveloping algebra of a dg Lie algebra $\g$, we obtain that the category of dg $\g$-modules (which is isomorphic to $\U(\g)\MMod$) is Quillen equivalent to $\B(\U(\g))\CComod$ and also to $\CE_\bullet(\g)\CComod$.
 \subsection{Derived duality functors for modules and comodules} Given a dg algebra $A$, there is a derived duality functor on $D(A): M\mapsto d(M)$ where
 $d(M):=\RHom_A(M,A).$

Abusing the notation slightly, we will write $d(M)$ for an object in $D(A^{\op})$ as well as for a dg $A$-module representing it. Note that $d(M)$ is naturally a right dg $A$-module (owing to the fact that $A$ is a two-sided module over itself). Formally, $d$ is a contravariant functor $D(A)\to D(A^{\op})$.

An analogue of this functor for a dg coalgebra $C$ is not so straightforward. Given a dg $C$-comodule $N$, set \[d(N):=\holim_\alpha N^*_{\alpha}\]
where the homotopy limit in the model category of dg $C$-comodules is taken over finite-dimensional dg subcomodules of $N$. The functor $N\mapsto d(N)$ is a contravariant functor $D^{\co}(C)\to D^{\co}(C^{\op})$. Then the following result holds.
\begin{prop}\label{prop:one-sideddual}
For any augmented dg algebra $A$,	there is a commutative diagram of functors
\[
\xymatrix{D(A)\ar_\cong[d]\ar^d[r]&D(A^{\op})^{\ar^\cong[d]}\\
D^{\co}(\B(A))\ar^d[r]&D^{\co}\big((\B (A))^{\op}\big)	
}
\]
where the vertical arrows are Koszul duality equivalences.
\end{prop}
\begin{proof} This is \cite[Proposition 3.11]{BCL}.\end{proof}
In other words, the derived duality functors $d$ for modules and comodules correspond to each other under Koszul duality (or, more accurately, Koszul equivalence).

Following \cite{BCL}, we will call that an augmented dg algebra $A$ is $n$-Gorenstein, if there is a quasi-isomorphism $d(\bk)[n]\simeq \bk$, i.e. if the $A$-module $\bk$ is self-dual up to a shift $n$. According to Proposition \ref{prop:one-sideddual}, this is equivalent to requiring an existence of a weak equivalence of $\B(A)$-comodules $d(\B(A))[n]$ and $\B(A)$; in other words (in the language of \cite[Theorem 8.3]{BCL}), $\B(A)$ is an $n$-Frobenius coalgebra.
\begin{defi}\label{def:dualizing}
	The \emph{dualizing complex} for an augmented dg algebra $A$ is the right dg $A$-module $\bd:=d(\bk)$.
	
If	$A$ is a Gorenstein dg algebra, its dualizing complex (whose cohomology is concentrated in a single degree and is one-dimensional)  will be called the \emph{dualizing module}.
\end{defi}
\begin{rem}
	If $A$ is a dg Hopf algebra, which is the case treated in detail in the next section, then the dualizing complex for $A$ is also called the \emph{homological left integral} for $A$. A detailed study of homological left integrals for ungraded Hopf algebras was undertaken in \cite{BZ}.
\end{rem}

\section{Duality for dg Hopf algebras}\label{section:Hopf}
Now let $A$ be a dg Hopf algebra with comultiplication $\Delta:A\to A\otimes A$ and the antipode $S$; not necessarily commutative or cocommutative.
Recall that  given two dg $A$-modules $M$ and $N$, their tensor product $M\otimes N$ is a dg $A$-module where the $A$-module action is given by the formula
\[
a(m\otimes n)=\Delta(a)(m\otimes n)
\]
for $a\in A, m\in M, n\in N$. Given a right $A$-module $M$, it can be viewed as a left module according to the formula $am:=(-1)^{|a||m|}mS(a)$ where $a$ and $m$ are homogeneous elements in $A$ and $M$ respectively. Given two dg $A$-modules $M$ and $N$, their $\bk$-linear hom space $\Hom(M,N)$ is naturally an $A\otimes A^{\op}$-module and it therefore has the structure of an $A$-module using the dg algebra map
\[
(1\otimes S)\circ\Delta:A\to A\otimes A^{\op}.
\]
Next, given two $A$-modules $M$ and $N$, there is an $A$-module map
\begin{equation}\label{eq:dualityiso}
\phi:N\otimes M^*\to \Hom(M,N)
\end{equation}
so that for $n\in N, m\in M, \alpha\in M^*$ we have  $\phi(n\otimes \alpha)(m)=n\alpha(m)$. If $\dim M<\infty$, then $\phi$ is an isomorphism. This is proved in \cite{Wi} for ungraded Hopf algebras but the proof carries over verbatim to the dg case.

Recall from  \cite[Section 8]{Pos} that for a dg algebra $A$ the category of
dg $A$-modules has an \emph{injective} model structure where weak equivalences are, as before, quasi-isomorphisms but cofibtations are injections. In this model structure all objects are cofibrant and the identity functor is a Quillen equivalence between the projective and injective model structures on $A\MMod$. We will refer to an injectively fibrant dg module as simply fibrant and to a projectively cofibrant one as simply cofibrant.
\begin{lem}\label{lem:inj}
	Let $A$ be a dg Hopf algebra, $M$ be a cofibrant dg $A$-module and $N$ be a fibrant dg $A$-module. Then $\Hom(M,N)$ is a fibrant dg $A\otimes A^{\op}$-module where the left $A$-action is induced from the left $A$-action on $N$ and the right $A$-action which is from the left $A$-action on $M$.
\end{lem}
\begin{proof}
Given an acyclic injection $X\to Y$ of $A\otimes A^{\op}$-modules, and a map $X\to \Hom(M,N)$ we need to show that the following diagram admits a lift:
\begin{equation}\label{eq:lift}
\xymatrix{
X\ar[d]\ar[r]&\Hom(M,N)\\Y\ar@{-->}[ur]
}
\end{equation}
	
This is equivalent to a lift in the following diagram of $A$-modules where the solid arrows are the adjoint to the solid arrows in (\ref{eq:lift}):
\[
\xymatrix{
	X\otimes_A M\ar[d]\ar[r]&N\\Y\otimes_A M\ar@{-->}[ur]
}
\]	
But since $M$ is cofibrant, the map $X\otimes_AM\to Y\otimes_AM$ is acyclic. It is also injective (this is obvious if $M$ is free as a graded $A$-module and follows in general). Since $N$ is a fibrant $A$-module, a desired lift exists.
\end{proof}
\begin{lem}\label{lem:hopf}
	Let $A$ be a dg Hopf algebra and $M,N$ be dg $A$-modules. Then there is a natural isomorphism in $D(\mathbf k)$:
	\begin{equation}\label{eq:RHom}{\RHom}_A(M,N)\to {\RHom}_A(\bk,\Hom(M,N)).\end{equation} If $N$ is an $A$-bimodule then ${\Hom}_A(M,N)$ and ${\Hom}_A(\bk,\Hom(M,N))$ are right $A$-modules and  (\ref{eq:RHom}) in an isomorphism in $D(A^{\op})$.
\end{lem}
\begin{proof}
	
	Without loss of generality, we assume that $M$ is cofibrant and $N$ is  fibrant dg $A$-modules. Then $\RHom_A(M,N)$ is represented by $\Hom_A(M,N)$.
	According to Lemma \ref{lem:inj}, $\Hom(M,N)$ is fibrant over $A\otimes A^{\op}$. Regarded as a dg $A$-module via the map $(1\otimes S)\circ \Delta:A\to A\otimes A $, it is also fibrant since the restriction functor is right Quillen and so preserves fibrant objects.
	Therefore, ${\RHom}_A(\bk,\Hom(M,N))$ is represented by $${\Hom}_A(\bk,\Hom(M,N))\cong \Hom(M,N)^A.$$

Next, dg vector spaces $\Hom_A(M,N)$ and $\Hom(M,N)^A$ are isomorphic, cf. \cite[Lemma 9.2.2]{Wi} where this statement is proved for ordinary Hopf algebras but the argument carries over to the dg graded case without any changes.

Finally, if $N$ has an additional structure of a right $A$-module, commuting with its left $A$-module structure, then all the isomorphisms considered above are clearly those of right $A$-modules.

\end{proof}
We need the following important result on modules over dg Hopf algebras.
\begin{lem}\label{lem:action}
	Let $N$ be a dg module over a dg Hopf algebra $A$ with the left action $m:A\otimes N\to N$ and denote by $N_{\operatorname{tr}}$ be the dg $A$-module whose underlying dg vector space is $N$ but the action is trivial, i.e. through the counit map $A\to\bk$.Then the map
	\[
	(1\otimes m)(1\otimes S\otimes 1)(\Delta\otimes 1):A\otimes N\to A\otimes N_{\operatorname{tr}}
	\]
	is an isomorphism of left dg $A$-modules. The inverse map is given by
	\[
	(1\otimes m)(\Delta\otimes 1):A\otimes N_{\operatorname{tr}}\to A\otimes N.
	\]
	The same pair of mutually inverse maps determines an isomorphism between $A\otimes N$ and $A\otimes N_{\operatorname{tr}}$ viewed as \emph{right} dg $A$-modules.
\end{lem}	
\begin{proof}
	This is proved in \cite[Lemma 9.2.9]{Wi} in the ungraded case but the proof generalizes to the dg case in an obvious fashion.
\end{proof}
We will now describe duality for proper dg modules over a dg Hopf algebra; recall that $\bd$ stands for the dualizing complex, cf. Definition \ref{def:dualizing}.
\begin{prop}\label{prop:duality}
Let $M$ be a proper dg module over a dg Hopf algebra $A$. Then there is an isomorphism in $D(A^{\op})$:
\[
d(M)\cong \Hom(M,\bd).
\]
\end{prop}
\begin{proof}
	Using Lemma \ref{lem:hopf} and finite-dimensionality of the cohomology of $M$, we have an isomorphism of dg vector spaces.
	\begin{align*}
	d(M)&={\RHom}_A(M, A)\\&\cong {\RHom}_A(\bk, \Hom(M,A))
	\\&\cong{\RHom}_A(\bk, A\otimes M^*).
	\end{align*}
	Next,  by Lemma \ref{lem:action},
	$A\otimes M^*\cong A\otimes (M^*)_{\operatorname{tr}}$ and so,
	\[
	{\RHom}_A(\bk, A\otimes (M^*)_{\operatorname{tr}})\cong \Hom(M, {\RHom}_A(\bk,A))=\Hom(M,\bd).
	\]
	Moreover, $A\otimes M^*$ is also a \emph{right} dg $A$-module with the right action of $A$ on itself and the trivial action on $M^*$. This right module structure is carried by the isomorphism	$A\otimes M^*\to A\otimes (M^*)_{\operatorname{tr}}$ of Lemma \ref{lem:action} to the diagonal right action of $A$ on $A\otimes(M^*)_{\operatorname{tr}}$. Thus, the obtained isomorphism $d(M)\cong  \Hom(M,\bd)$ is one in $D(A^{\op})$ rather than merely in $D(\bk)$ as claimed.
	\end{proof}
\subsection{Open Serre duality and open Calabi-Yau structures}
Let $\mathcal{C}$ be a $\bk$-linear triangulated category and  $\mathcal{L}, \mathcal{M}$ be two triangulated subcategories of $\mathcal{C}$. The pair $(\mathcal{L},\mathcal{M})$ determines \emph{open Serre duality} on $\mathcal{C}$  if for all $L\in\mathcal{L}$ and $M\in\mathcal{M}$, the graded hom
$$
\Hom^\bullet(M,L):=\prod_{i=-\infty}^{+\infty}\Hom^i(M,L),~\Hom^i(M,L)={\mathcal C}(M,L[i])
$$
is a finite-dimensional graded vector space and there is an auto-equivalence $F:\mathcal M\to \mathcal M$ together with the degree $0$  bifunctorial isomorphisms of graded vector spaces:
\begin{eqnarray}\label{Calabi-Yau-BCL}
		\Phi_{M,L}:\Hom^\bullet(M,L)\to\Hom^\bullet(L,F(M))^*,~~M\in\mathcal{M},~~L\in\mathcal{L}.
\end{eqnarray}
The functor $F$ is called the \emph{open Serre functor}.
A special case of a category with an open Serre functor is that of an open $n$-CY category:
\begin{defi}\label{defi:nCY}{\rm \cite[Definition 6.10]{BCL}}
An \emph{open  $n$-CY structure} on a triangulated category $\huaC$ is a pair $(\mathcal{L},\mathcal{M})$ of triangulated subcategories of $\huaC$ with open Serre duality as above where $F=[n]$, the $n$th iteration of the shift functor.
\end{defi}

\begin{rem}
In the case $\mathcal C=\mathcal{L}=\mathcal{M}$, 	the definition of an open $n$-CY structure is called an $n$-CY structure (without the adjective `open') and it is equivalent to the one given in \cite[Section 1.1]{Kel}. In particular, it is stronger than the notion of a \emph{weak} $n$-CY structure [Proposition 2.2 (a)]\cite{Kel-2008} (which we have no occasion to use). The proof of this equivalence is given in Appendix \eqref{A}.
\end{rem}

We will now specialize to $\mathcal{C}=D(A)$, the derived category of a dg Hopf algebra $A$, $\mathcal{L}=\Perf(A)$,  the subcategory of $D(A)$ consisting of perfect dg $A$-modules, and $\mathcal{M}=\operatorname{pvd}(A)$, the subcategory of $D(A)$ consisting of proper dg $A$-modules. The following result holds.

\begin{theorem}\label{thm:Serre}
	Let $L$ be a perfect dg $A$-module for a dg Hopf algebra $A$ and $M$ be a proper dg $A$-module. Then there is a natural isomorphism in $D(A)$:
	\[
{\RHom}^\bullet_A(M,L)\cong	{\RHom}^\bullet_A(L,M\otimes \bd^*)^*.
	\]
	Thus, $D(A)$ has open Serre duality with respect to the pair $(\operatorname{Perf}(A),\operatorname{pvd}(A))$ and the open Serre functor {$?\mapsto ?\otimes \bd^*$.}
\end{theorem}

\begin{proof}
	Since $L$ is perfect and $M$ is proper, we have
	\begin{align*}{\RHom}_A(M,L)&\cong d(M)\otimes_AL\\
		&\cong \Hom(M, \bd)\otimes_AL\\
		&\cong (M^*\otimes \bd)\otimes_AL
	\end{align*}	
	where the second isomorphism uses Proposition \ref{prop:duality}. Next, by the tensor-hom adjunction, we have
	\begin{align*}
	{\RHom}_A(M,L)^*&\cong{\big((M^*\otimes \bd)\otimes_AL\big)^*}\\
&\cong {\RHom}_A(L,(M^*\otimes \bd)^*)\\
	&\cong {\RHom}_A(L,M^{**}\otimes \bd^*)\\
	&\cong {\RHom}_A(L,M\otimes \bd^*).
	\end{align*}
Thus, we deduce that ${\RHom}^\bullet_A(M,L)\cong	{\RHom}^\bullet_A(L,M\otimes \bd^*)^*$.	This finishes the proof.
\end{proof}

As a consequence, we obtain that the derived category of a Gorenstein Hopf dg algebra has an open Calabi-Yau structure.
\begin{cor}\label{cor:HopfCY}
	Let $A$ be a dg Hopf algebra that is $n$-Gorenstein, i.e. $\RHom_A(\bk,A)\cong \bk[-n]$. Then the pair $(\Perf(A),\operatorname{pvd}(A))$ is an open $n$-Calabi-Yau structure on $D(A)$.
\end{cor}
\begin{proof}
	This follows immediately from Theorem \ref{thm:Serre}.
\end{proof}
\section{Dualizing complexes for dg Lie algebras}\label{sec:dualizing}
In this section we let $A=\U(\g)$, the universal enveloping algebra of a dg Lie algebra. We will assume that $\g$ is proper, i.e. that $\HH^\bullet(\g)$ is a totally finite-dimensional graded vector space. The dualizing complex for $\U(\g)$ is, therefore the dg right $\U(\g)$-module $\RHom_{\U(\g)}(\bk, \U(\g))\cong\CE^\bullet(\g,\U(\g))$ where the right action is induced by the action of $\U(\g)$ on itself by right multiplication. Then $\U(\g)$ is a dg Hopf algebra, and so, results of 	Section \ref{section:Hopf} apply. However, in this situation we can obtain much more precise results than for general dg Hopf algebras, and in the case when $\g$ is a graded Lie algebra (i.e. with a vanishing differential), further strengthening will be achieved.

Let $V$ be a graded vector space  of finite total dimension (so that the sum of dimensions of the graded components of $V$ is finite). Let us choose a basis $e_1,\ldots,e_n,\varepsilon_1,\ldots, \varepsilon_m$ of $V$ where $e_i,i=1,\ldots,n$ are even elements of $V$ and $\varepsilon_j,j=1,\ldots, m$ are odd elements. Set $|V|:=n-|e_1|-\ldots-|e_n|+|\varepsilon_1|+\ldots+|\varepsilon_m|$ and note that the integer $|V|$ does not depend on the choice of  basis in $V$. Note also that $|V|$ is \emph{not} a quasi-isomorphism invariant; we owe this elementary but important observation to an anonymous referee.
Specifically, write $V\cong \HH^\bullet(V)\oplus V_1\oplus V_2$ where $V_1$ is a contractible dg vector space with $d(V_1)$ situated in even degrees and $V_2$ is a contractible dg vector space with $d(V_2)$ being odd. Then $|V_1|=0$ and $|V_2|$ is an even number. Moreover,  $|V|=|\HH^\bullet(V)|+|V_1|+|V_2|=|\HH^\bullet(V)|+|V_2|$. In particular,  $|V|=|\HH^\bullet(V)| \mod 2$ and $|V|=|\HH^\bullet(V)|$ provided $V_2=0$.

 Then we have the following result.

\begin{theorem}\label{the:1-dim} Let $\g$ be a proper dg Lie algebra.
	Then the dualizing complex $\CE^\bullet(\g, \U(\g))$ has 1-dimensional cohomology concentrated in the total degree $|\HH^\bullet(\g)|$.
\end{theorem}
\begin{proof}
	Firstly, using the spectral sequence
	\[
	\HCE^\bullet\Big(\HH^\bullet(\g),\U(\HH^\bullet(\g))\Big)\Rightarrow \HCE^\bullet(\g, \U(\g)),
	\]
we reduce the claim to the special case when $\g$ is a graded Lie algebra (with vanishing differential, so $\g=\HH^\bullet(\g)$).	

Next, we have the following identification of graded vector spaces
	\[
	\CE^\bullet(\g, \U(\g))\cong \hat{\huaS}(\g^*[-1])\hat{\otimes}\U(\g).
	\]
	Introduce the weight filtration on $\hat{\huaS}(\g^*[-1])\hat{\otimes}\U(\g)$ by giving an element in $\g^*[-1]\subset \hat{\huaS}(\g^*[-1])$ weight $1$ and an element in $\g\in \U(\g)$ weight $-1$.   Letting $F_p \hat{\huaS}(\g^*[-1])\hat{\otimes}\U(\g)$ be the subspace  in $\hat{\huaS}(\g^*[-1])\hat{\otimes}\U(\g)$ spanned by monomials of weight greater or equal than $-p$, we get an increasing filtration on $	\CE^\bullet(\g, \U(\g))$ of the form $\ldots\subset F_n\subset F_{n+1}\subset\ldots, n\in\mathbb{Z}$. This filtration is exhaustive and the differential in 	$\CE^\bullet(\g, \U(\g))$ is compatible with it. Moreover, since $\g$ is itself graded (nondifferential), the complex 	$\CE^\bullet(\g, \U(\g))$ has a bigrading, with the additional CE grading so that an element in 	$\hat{\huaS}^n(\g^*[-1])\hat{\otimes}\U(\g)$ has the CE grading $n$. Restricted to $\hat{\huaS}^n(\g^*[-1])\hat{\otimes}\U(\g)$, the  filtration $\{F_p\}$ is bounded below, namely, it runs as $0\subset F_{-n}\subset F_{-n+1}\subset\ldots$. Thus, the spectral sequence associated with the filtration $\{F_p\}$ converges strongly.
	
	Moreover, the $E_1$-page of the associated spectral sequence coincides with the CE cohomology of the \emph{abelian} graded Lie algebra with the same underlying graded vector space as $\g$. In other words, it computes
	 $\operatorname{Ext}_{\huaS(\g)}(\bk,\huaS(\g))$ where the symmetric algebra $\huaS(\g)$ acts on itself by the left multiplication. The claim now becomes a standard homological calculation with the Koszul complex.
\end{proof}

So, the dualizing complex $\bd:=d(\bk)$ for any proper dg Lie algebra $\g$ has one-dimensional cohomology; according to our conventions, we will call it the \emph{dualizing module} for $\g$. The $\g$-action (or, equivalently, $\U(\g)$-action) on $\bd$ is a map of dg Lie algebras $\g\to \End(V)$ where $V$ is a dg vector space that represents $\bd$ in $D(\U(\g))$. Equivalently, this is a map of dg algebras $\U(\g)\to\End(V)$. Since $\HH^\bullet(V)$ is 1-dimensional, the dg Lie algebra $\End(V)$ is quasi-isomorphic to $\bk$ with the zero bracket and, assuming without loss of generality that $\g$ is a cofibrant Lie algebra, we obtain that a $\g$-module structure on $\bd$ is realized by a dg Lie algebra map from $\g$ to the abelian Lie algebra $\bk$ or as a dg algebra map $\U(\g)\to \bk$. Another equivalent point of view, which  does not assume cofibrancy of $\g$, is that the structure of a $\g$-module on a one-dimensional space is determined by an $L_\infty$-map $\g\to\bk$ or an $A_\infty$-map $\U(\g)\to \bk$. This latter point of view will be used later on.

In any case, the data of an $L_\infty$-map $\g\to\bk$ are equivalent to that of an MC element in the pc dg algebra $\CE^\bullet(\g)$. Since $\CE^\bullet(\g)$ is graded commutative, this is further equivalent to a CE 1-cocycle of $\g$, and cohomologous cocycles give rise to homotopic $L_\infty$-maps (or different $\g$-modules representing the same object in $D(\U(\g))$. We have the following definition.
\begin{defi}
Given a proper dg algebra Lie $\g$, denote by $\chi(\g)\in\CE^1(\g)$ the CE cohomology class of $\g$ corresponding to the dualizing module $\bd$.
\end{defi}
\begin{prop}\label{D-dual-CE}
	Let $M$ be a proper dg $\g$-module where $\g$ is a proper dg Lie algebra and consider $M^!:=\CE_\bullet(\g,M)$. Then there is a weak equivalence of $\CE_\bullet(\g)$-comodules:
	\[
	d(M^!)\simeq \CE_\bullet(\g, \bd\otimes M^*).
	\]
	
	In particular (when $M$ is the trivial $\g$-module $\bk$), the comodule $d(\CE_\bullet(\g))$ is weakly equivalent to the comodule whose dual pc module is $\CE^\bullet(\g)^{[\chi(\g)]}$, the twist of $\CE^\bullet(\g)$ by the MC element $\chi(\g)$.
\end{prop}
\begin{proof}
	The formula for $d(M^!)$ follows directly from Proposition \ref{prop:duality}. This further implies the claimed description of $d(\CE_\bullet(\g))$.
\end{proof}		
 
\section{Dualizing complexes for graded Lie algebras}\label{sec:technical}
In this section we consider the case of a graded Lie algebra $\g$ (i.e. a dg Lie algebra with vanishing differential). The dualizing module $\g$-module $\bd(\g):=\RHom_{\U(\g)}(\bk,\U(\g))$ for $\g$ is identified explicitly, in complete analogy with the ungraded case. This requires a calculation in coordinates with the CE complex $\CE^\bullet(\g,\U(\g))$. Since the calculation in the graded case is quite involved, let us first describe the underlying strategy.

We have seen previously (Theorem \ref{the:1-dim}) that a spectral sequence argument shows that the cohomology of $\CE^\bullet(\g,\U(\g))$ is one-dimensional. Moreover, if $\g$ is abelian, one can write an explicit nontrivial cocycle (which is often referred to as the Berezinian of $\g$). The action of $\g$ on it is (in the abelian case), of course, trivial. The first step is to compute (in the non-abelian case) an explicit representative of a nontrivial cocycle of $\CE^\bullet(\g,\U(\g))$. This requires an application of homological perturbation theory and the result is what we call a `deformed Berezinian'. The next step is to compute the $\U(\g)$ action of the one-dimensional space spanned by the deformed Berezinian. For this, we use an explicit form of an $A_\infty$-minimal model of modules over an $A_\infty$-algebra $\U(\g)$. In our situation, $\U(\g)$ is an ordinary graded algebra but its modules can still support a nontrivial $A_\infty$-structure. Nevertheless, it turns out that no higher $A_\infty$-correction terms are present. This is due to the fact that $\U(\g)$ is a graded algebra (i.e. with vanishing differential) and so, $\CE^\bullet(\g,\U(\g))$ has a canonical bigrading. The last step is just to compute the lowest (and only) term of the $A_\infty$-action of $\U(\g)$ on the deformed Berezinian. This is where the trace of the adjoint action of $\g$ appears.

 We start by spelling out in detail the formulae for the CE differentials, including the sign conventions.

Let $\g$ be finite dimensional $\mathbb Z$-graded Lie algebra over the field $\bk$ and let $\{e_1,\ldots,e_n\}$ be a  homogeneous basis of $\g$. For $i=1,\ldots,n$, we denote $e_i[1]$ by $x_i$.
Then  $\{x_1,\ldots,x_n\}$ is a basis of $\g[1]$ and $\{x^1,\ldots,x^n\}$ the dual basis of $(\g[1])^*$, so that $\langle x^i,x_j\rangle=\delta_i^j$.

Moreover, we have the structure constants, $\{N_{ij}^k\}$  given by
\begin{eqnarray}\label{structural constant}
[e_i,e_j]=\sum_{k=1}^{n}N_{ij}^k e_k.
\end{eqnarray}
The CE coalgebra $\CE_\bullet(\g)$ has $\huaS^c(\g[1])$ as its underlying cofree conilpotent cocommutative coalgebra and the codifferential $\delta_{\CE}$ is determined by the formula
\begin{eqnarray}\label{codifferential}
\nonumber&&\delta_{\CE}(x_ix_j)=\sum_{k=1}^{n}(-1)^{|x_i|+1}N_{ij}^k x_k.
\end{eqnarray}
The CE algebra $\CE^\bullet(\g)=[\huaS^c(\g[1])]^*\cong\hat{\huaS}[(\g[1])^*]$ is the linear dual of the CE coalgebra $\CE_\bullet(\g)$, that is, the CE differential $d_{\CE}$ is given by
\begin{eqnarray}\label{CE-Differential}
\langle d_{\CE}(X),A\rangle=-(-1)^{|X|}\langle X,\delta_{\CE}(A)\rangle,\,\,\,\,\forall X\in\CE^\bullet(\g), A\in \CE_\bullet(\g).
\end{eqnarray}
Since the characteristic of the field $\bk$ is zero, the symmetric coinvariants  $[(\g[1])^{\otimes n}]_{\mathbb S_{n}}$ are isomorphic to the invariants $[(\g[1])^{\otimes n}]^{\mathbb S_{n}}$. More precisely, we can define mutually inverse maps \cite{HL}
\begin{eqnarray*}
i_n:[(\g[1])^{\otimes n}]_{\mathbb S_{n}}\to [(\g[1])^{\otimes n}]^{\mathbb S_{n}},\,\,\,\pi_n:[(\g[1])^{\otimes n}]^{\mathbb S_{n}}\to [(\g[1])^{\otimes n}]_{\mathbb S_{n}}
\end{eqnarray*}
by the formulae
\begin{eqnarray}
\label{iso-one}&&i_n(\xi_1\ldots\xi_n)=\sum_{\sigma\in \mathbb S_{n}}\varepsilon(\sigma;\xi_1,\ldots,\xi_n)\xi_{\sigma(1)}\otimes\ldots\otimes\xi_{\sigma(n)},\\
\label{iso-two}&&\pi_n(\xi_{1}\otimes\ldots\otimes\xi_{n})=\frac{1}{n!}\xi_1\ldots\xi_n.
\end{eqnarray}
By \eqref{CE-Differential}, we can assume that
\begin{eqnarray}\label{d_CE}
 d_{\CE}=\sum_{k=1}^n \sum_{1\le p, q\le n}\frac{1}{2} \alpha_{pq}^k x^p x^q\frac{\partial}{\partial x^k},~\alpha_{pq}^k=(-1)^{|x^p||x^q|}\alpha_{qp}^k\in \bk.
\end{eqnarray}
Let us express the above formula in terms of the constants $\{N^k_{pq}\}$ rather than  $\{\alpha^{k}_{pq}\}$. For $X=x^k$ and $A=x_ix_j$, we have
\begin{eqnarray*}
\langle d_{\CE}(x^k),x_ix_j\rangle&\stackrel{\eqref{d_CE}}{=}&\langle \sum_{1\le p,q\le n} \frac{1}{2} \alpha_{pq}^k x^p x^q,x_ix_j\rangle\\
&\stackrel{\eqref{iso-one}}{=}&\sum_{1\le p,q\le n} \frac{1}{2}\alpha_{pq}^k x^p x^q\big(x_i\otimes x_j+(-1)^{|x_i||x_j|}x_j\otimes x_i\big)\\
&{=}&\frac{1}{2}(-1)^{|x_i||x^j|}\alpha_{ij}^k x^i(x_i) x^j(x_j)+\frac{1}{2}(-1)^{|x_i||x_j|+|x_j||x^i|}\alpha_{ji}^k x^j(x_j) x^i(x_i)\\
&=&(-1)^{|x_i||x^j|}\alpha_{ij}^k.
\end{eqnarray*}
On the other hand,  we have
\begin{eqnarray*}
-(-1)^{|x^k|}\langle x^k,\delta_{\CE}(x_ix_j)\rangle&\stackrel{\eqref{codifferential}}{=}&(-1)^{|x^k|+1}\langle x^k,\sum_{p=1}^{n}(-1)^{|x_i|+1}N_{ij}^p x_p\rangle\\
&{=}&(-1)^{|x^k|+|x_i|}N_{ij}^k.
\end{eqnarray*}
We deduce that $\alpha_{ij}^k=(-1)^{1+(|x_i|+1)|x_j|}N_{ij}^k$. It follows that
\begin{eqnarray}\label{g-Lie-to-dCE}
 d_{\CE}=-\frac{1}{2}\sum_{k=1}^n \sum_{1\le p,q\le n} (-1)^{(|x_p|+1)|x_q|}N_{pq}^k x^p x^q\frac{\partial}{\partial x^k}.
\end{eqnarray}

From now on, we will distinguish the even and odd elements of the basis of the graded Lie algebra $\g$.  We assume that $\{e_1,\ldots,e_n;\varepsilon_{1},\ldots,\varepsilon_{m}\}$ is a  homogeneous basis of $\g$ and with elements $\{e_1,\ldots,e_n\}$ even and $\{\varepsilon_{1},\ldots,\varepsilon_{m}\}$  odd. We denote $e_i[1]$ by $x_i$ and $\varepsilon_i[1]$ by $y_i$ respectively. Then  $\{x_1,\ldots,x_n;y_1,\ldots,y_m\}$ is a  basis of $\g[1]$ and $\{x^1,\ldots,x^n;y^1,\ldots,y^m\}$ the dual basis of $(\g[1])^*$, which is associated with the basis $\{x_1,\ldots,x_n;y_1,\ldots,y_m\}$ of $\g[1]$.

Moreover, we have three families $\{a_{ij}^k\}$, $\{b_{ij}^k\}$ and $\{c_{ij}^k\}$ of structure constants  given by the following formulae:
\begin{eqnarray}\label{structural constant-1}
[e_i,e_j]=\sum_{k=1}^{n}a_{ij}^k e_k,\,\,[e_i,\varepsilon_j]=\sum_{k=1}^{m}b_{ij}^k \varepsilon_k,\,\,[\varepsilon_i,\varepsilon_j]=\sum_{k=1}^{n}c_{ij}^k e_k.
\end{eqnarray}

By \eqref{g-Lie-to-dCE}, we deduce that
{\small{
\begin{eqnarray}\label{DCE-diff}
d_{\CE}=-\frac{1}{2}\sum_{k=1}^n \left\{\sum_{1\le p,q\le n} a_{pq}^k x^p x^q+\sum_{1\le p,q\le m} c_{pq}^k y^p y^q\right\}\frac{\partial}{\partial x^k}-\sum_{k=1}^m \left\{\sum_{1\le p\le n\atop  1\le q\le m} b_{pq}^k x^py^q\right\}\frac{\partial}{\partial y^k}.
\end{eqnarray}
}}

Let us consider the  object $\RHom_{\U(\g)}(\bk,\U(\g))$ in $D(\U(\g)^{\op})$. It is represented by a right $\U(\g)$-module $\CE^\bullet(\g, \U(\g))$ with the underlying graded vector space is $\hat{\huaS}[(\g[1])^*]\hat{\otimes} \U(\g)$. The differential  $\delta$ and the right $\U(\g)$-module structure on $\hat{\huaS}[(\g[1])^*]\hat{\otimes} \U(\g)$ are given by
\begin{eqnarray*}
\delta(A\otimes B)&=&d_{\CE}(A)\otimes B+\sum_{i=1}^{n}x^iA\otimes e_i B+\sum_{i=1}^{m}{(-1)^{|A|}}y^i A\otimes \varepsilon_i B,\\
       (A\otimes B).u&=&A\otimes Bu,~\forall A\otimes B\in \hat{\huaS}[(\g[1])^*]\hat{\otimes} \U(\g),~u\in \g.
\end{eqnarray*}

We will use the homological interpretation of the Berezinian and the homological perturbation lemma to compute the right action of  $\U(\g)$ on $\RHom_{\U(\g)}(\bk,\U(\g))$. Recall that the PBW isomorphism $\Phi: \huaS(\g)\to \U(\g)$ is given by the following formula:
\begin{eqnarray}\label{PBW}
\Phi(u_1\ldots u_k)=\sum_{\sigma\in \mathbb S_{k}}\frac{\varepsilon(\sigma;u_1,\ldots,u_k)}{k!}u_{\sigma(1)}\ldots u_{\sigma(k)},~k=0,1,\ldots,~u_1,\ldots, u_k\in\g.
\end{eqnarray}
Using this isomorphism, one can define the Gutt star product $*$ on $\huaS(\g)$ which is the deformation quantization  of the Lie-Poisson algebra $\huaS(\g)$, \cite{Gutt}. The Gutt star product is given by
\begin{eqnarray}\label{Gutt}
\nonumber X*X'&=&\Phi^{-1}\Big(\Phi(X)\Phi(X')\Big)=\sum_{i=1}^{p+q}\pr_i\Big(\Phi^{-1}\big(\Phi(X)\Phi(X')\big)\Big)\\
    &=&XX'+\frac{1}{2}[X,X']+\sum_{i=1}^{p+q-2}\pr_i\Big(\Phi^{-1}\big(\Phi(X)\Phi(X')\big)\Big),
\end{eqnarray}
where $X\in \huaS^p(\g), X'\in \huaS^q(\g)$ and $\pr_i: \huaS(\g)\to \huaS^i(\g)$ is the projection onto the homogeneous part of polynomial degree $i$.

Using the isomorphism $\id\otimes\Phi$, one can transfer the right  dg $\U(\g)$-module structure from $\hat{\huaS}[(\g[1])^*]\hat{\otimes} \U(\g)$ to $\hat{\huaS}[(\g[1])^*]\hat{\otimes} \huaS(\g)$.
The differential operator $\partial$ on $\hat{\huaS}[(\g[1])^*]\hat{\otimes} \huaS(\g)$ is defined by
\begin{eqnarray}\label{Gutt-differential}
\nonumber\partial(A\otimes X)&=&\Big((\id\otimes \Phi^{-1})\delta(\id\otimes \Phi)\Big)(A\otimes X)\\
\nonumber&=&d_{\CE}(A)\otimes X+\sum_{i=1}^{n}x^iA\otimes \Phi^{-1}\big(e_i \Phi(X)\big)+\sum_{i=1}^{m}(-1)^{|A|}y^i A\otimes \Phi^{-1}\big(\varepsilon_i \Phi(X)\big)\\
&=&d_{\CE}(A)\otimes X+\sum_{i=1}^{n}x^iA\otimes e_i * X+\sum_{i=1}^{m}{(-1)^{|A|}}y^i A\otimes \varepsilon_i * X.
\end{eqnarray}
The right $\U(\g)$-module structure on $\hat{\huaS}[(\g[1])^*]\hat{\otimes} \huaS(\g)$ is defined by
\begin{eqnarray}\label{Gutt-representation}
\nonumber(A\otimes X)_*u&=&(\id\otimes \Phi^{-1})\Big([(\id\otimes \Phi)(A\otimes X)].u\Big)=A\otimes \Phi^{-1}\big(\Phi(X)u\big)\\
&=&A\otimes Xu+\frac{1}{2}A\otimes [X,u]+\sum_{i=1}^{p-1}A\otimes\pr_i\Big(\Phi^{-1}\big(\Phi(X)u\big)\Big).
\end{eqnarray}

Let us now recall the notion of an  abstract Hodge decomposition on a dg vector space, {\rm \cite[Definition 2.1]{CL}}.
\begin{defi}\label{Hodge}
An {\bf abstract Hodge decomposition} on a dg vector space $(V,d)$ is a pair of operators $s$ and $t$ on $V$
\begin{eqnarray}
s:V\to V,~~~~\,\,\,\, t:V\to V
\end{eqnarray}
of degrees $-1$ and $0$ respectively, such that
\begin{itemize}
			
			\item[(1)]$ds+sd=\id-t$,

            \item[(2)]$d t=t d$,

            \item[(3)]$t^2=t$,

            \item[(4)]$s^2=st=ts=0$.
		\end{itemize}
\end{defi}

We define $d:\hat{\huaS}[(\g[1])^*]\hat{\otimes} \huaS(\g)\to \hat{\huaS}[(\g[1])^*]\hat{\otimes} \huaS(\g)$ as follows:
\begin{eqnarray}
d(A\otimes X)=\sum_{i=1}^{n}x^iA\otimes e_i X+\sum_{i=1}^{m}(-1)^{|A|}y^i A\otimes \varepsilon_i X.
\end{eqnarray}
The Laplacian operator $\Delta:\hat{\huaS}[(\g[1])^*]\hat{\otimes} \huaS(\g)\to \hat{\huaS}[(\g[1])^*]\hat{\otimes} \huaS(\g)$ is defined by
\begin{eqnarray*}
\Delta=\sum_{i=1}^{n}\frac{\partial}{\partial e_i}\frac{\partial}{\partial x^i}+\sum_{i=1}^{m}\frac{\partial}{\partial y^i}\frac{\partial}{\partial \varepsilon_i}.
\end{eqnarray*}
For homogeneous polynomials
 $P\in \bk[e_1,\ldots,e_n,y^1,\ldots,y^m], Q\in \Lambda(x^1,\ldots,x^n,\varepsilon_1,\ldots,\varepsilon_m)$, we denote the polynomial degrees  of  $P,~Q$ by $\deg(P)$ and $\deg(Q)$ respectively. The following result holds.
\begin{prop}
For homogeneous polynomials $P, Q$, the following equation holds
\begin{eqnarray}\label{Berezinian}
[\Delta,d](P\otimes Q)=\big(m+n+\deg(P)-\deg(Q)\big)(P\otimes Q).
\end{eqnarray}
\end{prop}
\begin{proof}
	This is \cite[Proposition 4.2]{Covolo}.
\end{proof}
\begin{rem}
Since $m+n-\deg(Q)\ge 0$ and $\deg(P)\ge 0$, it follows  that $m+n+\deg(P)-\deg(Q)=0$ if and only if $\deg(P)=0$ and $\deg(Q)=m+n$. Moreover, there is a following isomorphism, cf. \cite[Proposition 4.2]{Covolo}:
$$
\HH^\bullet(\hat{\huaS}[(\g[1])^*]\hat{\otimes}  \huaS(\g),d)\cong \bk\{x^1\ldots x^n\otimes \varepsilon_1\ldots\varepsilon_m\}.
$$
Note that $\HH^\bullet(\hat{\huaS}[(\g[1])^*]\hat{\otimes}  \huaS(\g),d)\cong \operatorname{Ext}^\bullet_{\huaS(\g)}(\bk, \huaS(\g))$. This one-dimensional vector space (or its basis vector $x^1\ldots x^n\otimes \varepsilon_1\ldots\varepsilon_m$) is called the \emph{Berezinian} (or superdeterminant) of the graded vector space $\g$; this is the standard cohomological interpretation of the graded Berezinian \cite{Covolo, manin, delmorgan}. In our situation, $\g$ is a graded Lie algebra, and there is a $\g$-action on the Berezinian of $\g$. The calculation of this action is the main result of this section.
\end{rem}
We denote by $t:\hat{\huaS}[(\g[1])^*]\hat{\otimes}  \huaS(\g)\to \hat{\huaS}[(\g[1])^*]\hat{\otimes}  \huaS(\g)$ the operator of projection onto the one-dimensional vector space $\bk\{x^1\ldots x^n\otimes \varepsilon_1\ldots\varepsilon_m\}$ (usually called the Berezinian module for $\g$). Furthermore, define the operator
$s:\hat{\huaS}[(\g[1])^*]\hat{\otimes}  \huaS(\g)\to \hat{\huaS}[(\g[1])^*]\hat{\otimes}  \huaS(\g)$ by the following formula:
\begin{eqnarray}\label{homotopy}
	s(P\otimes Q)=\left\{
	\begin{array}{ll}
		\frac{\Delta(P\otimes Q)}{m+n+\deg(P)-\deg(Q)}, &\mbox {$P\otimes Q\notin \bk\{x^1\ldots x^n\otimes \varepsilon_1\ldots\varepsilon_m\}$, }\\
		\Delta(P\otimes Q), &\mbox {$P\otimes Q\in \bk\{x^1\ldots x^n\otimes \varepsilon_1\ldots\varepsilon_m\}$.}
	\end{array}
	\right.
\end{eqnarray}

\begin{prop}
In the above notation, $\Big(\hat{\huaS}[(\g[1])^*]\hat{\otimes}  \huaS(\g),d,s,t\Big)$ is an abstract Hodge decomposition.
\end{prop}

\begin{proof}
By \cite[Proposition 4.1]{Covolo}, we have $d^2=0$. By \eqref{Berezinian}-\eqref{homotopy} and the definition of $t$, we obtain $ds+sd=\id-t$. Since $t$ is a projection map, it follows that $t^2=t$. Moreover, by
direct computation, we deduce that
\begin{eqnarray*}
d t=t d=s^2=st=ts=0.
\end{eqnarray*}
Thus, we prove that $\Big(\hat{\huaS}[(\g[1])^*]\hat{\otimes}  \huaS(\g),d,s,t\Big)$ is an abstract Hodge decomposition.
\end{proof}

We define $x:\hat{\huaS}[(\g[1])^*]\hat{\otimes}  \huaS(\g)\to \hat{\huaS}[(\g[1])^*]\hat{\otimes} \huaS(\g)$ as following:
\begin{eqnarray}\label{differential-x}
&&x(A\otimes X)=d_{\CE}(A)\otimes X+\frac{1}{2}\sum_{i=1}^{n}x^iA\otimes [e_i,  X]+\frac{1}{2}\sum_{i=1}^{m}(-1)^{|A|}y^iA\otimes [\varepsilon_i,  X]\\
\nonumber&&+\sum_{i=1}^{n}\sum_{s=1}^{p-1}x^iA\otimes\pr_s\Big(\Phi^{-1}\big(e_i\Phi(X)\big)\Big)+\sum_{i=1}^{m}\sum_{s=1}^{p-1}(-1)^{|A|}y^iA\otimes\pr_s\Big(\Phi^{-1}\big(\varepsilon_i\Phi(X)\big)\Big),
\end{eqnarray}
here $X\in \huaS^p(\g)$. By \eqref{Gutt}, we obtain that $\partial=d+x$. It means that  $x$ is a perturbation of the differential  $d$.

\begin{lem}
With the notation as above, $\id+sx$ and $\id+xs$ are invertible maps.
\end{lem}
\begin{proof}
For any $A\in \hat{\huaS}[(\g[1])^*], X\in\huaS^p(\g)$, we have
$$
(sx)^{p+1}(A\otimes X)=(xs)^{p+1}(A\otimes X)=0.
$$
It follows that $sx$ and $xs$ are locally nilpotent maps. Thus, we deduce that $\id+sx$ and $\id+xs$ are invertible and
\begin{eqnarray*}
(\id+sx)^{-1}=\sum_{k=0}^{+\infty}(-sx)^k,~~~~~(\id+xs)^{-1}=\sum_{k=0}^{+\infty}(-xs)^k.
\end{eqnarray*}
\end{proof}
We define the operators $\alpha$ and $\beta$ on $\hat{\huaS}[(\g[1])^*]\hat{\otimes}  \huaS(\g)$  by
\begin{equation}\label{eq:alpha}
\alpha=(\id+sx)^{-1} \quad \text{and} \quad \beta=(\id+xs)^{-1}
\end{equation}
\begin{prop}
$\Big(\hat{\huaS}[(\g[1])^*]\hat{\otimes}  \huaS(\g),\partial=d+x,\alpha s,\alpha t\beta\Big)$ is an abstract Hodge decomposition and
\begin{eqnarray}\label{iso-HPL}
\alpha t :\big(t\big\{\hat{\huaS}[(\g[1])^*]\hat{\otimes}  \huaS(\g)\big\},0\big)\to \big((\alpha t\beta)\big\{\hat{\huaS}[(\g[1])^*]\hat{\otimes}  \huaS(\g)\big\},d+x\big)
\end{eqnarray}
is an isomorphism of dg vector spaces.
\end{prop}

\begin{proof}
It follows immediately from the Homological Perturbation Lemma \cite[Corollary 3.7]{CL}.
\end{proof}

By $t\big\{\hat{\huaS}[(\g[1])^*]\hat{\otimes}  \huaS(\g)\big\}=\bk\{x^1\ldots x^n\otimes \varepsilon_1\ldots\varepsilon_m\}$ and \eqref{iso-HPL}, we deduce that
$$
(\alpha t\beta)\big\{\hat{\huaS}[(\g[1])^*]\hat{\otimes}  \huaS(\g)\big\}=\bk\{\alpha(x^1\ldots x^n\otimes \varepsilon_1\ldots\varepsilon_m)\}.
$$

Since $\Big(\hat{\huaS}[(\g[1])^*]\hat{\otimes}  \huaS(\g),\partial=d+x,\alpha s,\alpha t\beta\Big)$ is an abstract Hodge decomposition, we have the following result.
\begin{prop}
There are two morphisms of dg vector spaces
\begin{eqnarray}\label{deformation-retract}
i: \big(\bk\{\alpha(x^1\ldots x^n\otimes \varepsilon_1\ldots\varepsilon_m)\},0\big)\rightleftharpoons \big(\hat{\huaS}[(\g[1])^*]\hat{\otimes} \huaS(\g),\partial=d+x\big): \alpha t\beta
\end{eqnarray}
such that $\alpha t\beta i=\id$ and $\id-i\alpha t\beta=(d+x)\alpha s+\alpha s(d+x)$. Here $i$ is the natural embedding map.
\end{prop}


Since the one-dimensional vector space $\big(\bk\{\alpha(x^1\ldots x^n\otimes \varepsilon_1\ldots\varepsilon_m)\}\big)$ (with the zero differential) is a deformation retract of  $\big(\hat{\huaS}[(\g[1])^*]\hat{\otimes} \huaS(\g),\partial=d+x\big)$, one can transfer the right $A_\infty$-module structure over $\U(\g)$ from $\hat{\huaS}[(\g[1])^*]\hat{\otimes}  \huaS(\g)$ to $\bk\{\alpha(x^1\ldots x^n\otimes \varepsilon_1\ldots\varepsilon_m)\}$. The higher $A_\infty$-action maps are calculated in terms of summations over rooted planar trees, cf. \cite[Proposition 3.18]{CL}. Note that the trees are in our situation binary since $\U(\g)$ is a graded algebra rather than an $A_\infty$-algebra. For the same reason, the dg vector space
\[
\CE^\bullet(\g,\U(\g))\cong \big(\hat{\huaS}[(\g[1])^*]\hat{\otimes} \huaS(\g),\partial=d+x\big)
\]
has a bigrading -- the total grading and the CE-grading. The operator  $x$  increases  the CE-degree by 1 and $s$ reduces the CE-degree by $1$. Since the trees involved in the formulas for the higher $A_\infty$-action have edges decorated by the operator {$s$}, and so, they all reduce the CE-degree. Since the $A_\infty$-module $\big(\bk\{\alpha(x^1\ldots x^n\otimes \varepsilon_1\ldots\varepsilon_m)\}\big)$ is one-dimensional (and so, is concentrated in a single CE-degree), all higher $A_\infty$-action maps on it vanish. Thus, that there is only a right $\U(\g)$-module structure on $\big(\bk\{\alpha(x^1\ldots x^n\otimes \varepsilon_1\ldots\varepsilon_m)\},0\big)$ and the higher $A_\infty$-structure maps are zero.

Let us denote the element  $x^1\ldots x^n\otimes \varepsilon_1\ldots\varepsilon_m$ by $\huaB$ (for `Berezinian'). We have also the \text{\emph{deformed Berezinian}}:
\begin{eqnarray}\label{deformed-Berezinian}
\,\,\,\, \widetilde{\huaB}:=\alpha(x^1\ldots x^n\otimes \varepsilon_1\ldots\varepsilon_m).
\end{eqnarray}

The right dg  $\U(\g)$-module structure on $\big(\bk\{\alpha(x^1\ldots x^n\otimes \varepsilon_1\ldots\varepsilon_m)\},0\big)$ is given by
\begin{eqnarray}\label{dg-U-mod}
\widetilde{\huaB}.u=(\alpha t\beta)(\widetilde{\huaB}_*u),~~u\in \g.
\end{eqnarray}
Since $\dim\bk\{\alpha(x^1\ldots x^n\otimes \varepsilon_1\ldots\varepsilon_m)\}=1$, we deduce that $\widetilde{\huaB}.u=r \widetilde{\huaB},~r\in\bk$. Thus, we have
\begin{eqnarray}
(\alpha t\beta)(\widetilde{\huaB}_*u)=r \widetilde{\huaB},~u\in\g.
\end{eqnarray}

We introduce the \emph{weight grading} on $\hat{\huaS}[(\g[1])^*]\hat{\otimes}\huaS(\g)$ by giving an element in $(\g[1])^*\subset \hat{\huaS}[(\g[1])^*]$ weight $1$ and an element in $\g\in \huaS(\g)$ weight $-1$. Note that $x$ increases the weight by 1 or more, and $s$ preserves the weight.

\begin{lem}
The lowest weight term in the expression for $(\alpha t\beta)(\widetilde{\huaB}_*u)$ is
$$
t\Big\{\frac{1}{2}x^1\ldots x^n\otimes [\varepsilon_1\ldots\varepsilon_m,u]-\big((sx)(x^1\ldots x^n\otimes \varepsilon_1\ldots\varepsilon_m)\big)u-(xs)(x^1\ldots x^n\otimes \varepsilon_1\ldots\varepsilon_m u)\Big\}.
$$
Moreover, its weight is $n-m$.
\end{lem}
\begin{proof}
Since $t$ is the operator of projection onto the one-dimensional vector space $\bk\{x^1\ldots x^n\otimes \varepsilon_1\ldots\varepsilon_m\}$ and $sx, xs$ increase the weight by greater than or equal to 1, we deduce that the weight of the lowest weight term is $n-m$. By ignoring the  higher weight terms of $(\alpha t\beta)(\widetilde{\huaB}_*u)$, one can deduce that
{\small{
\begin{eqnarray*}
&&(\alpha t\beta)(\widetilde{\huaB}_*u)\\
&\stackrel{\eqref{eq:alpha}}{=}&(t\beta)\Big(\big((\id-sx+\ldots)(x^1\ldots x^n\otimes \varepsilon_1\ldots\varepsilon_m)\big)_*u\Big)\\
&=&(t\beta)\Big((x^1\ldots x^n\otimes \varepsilon_1\ldots\varepsilon_m)_*u-\big((sx)(x^1\ldots x^n\otimes \varepsilon_1\ldots\varepsilon_m)\big)_*u\Big)\\
&=&(t\beta)\Big(x^1\ldots x^n\otimes \varepsilon_1\ldots\varepsilon_m u+\frac{1}{2}x^1\ldots x^n\otimes [\varepsilon_1\ldots\varepsilon_m,u]-\big((sx)(x^1\ldots x^n\otimes \varepsilon_1\ldots\varepsilon_m)\big)u\Big)\\
&=&t\Big\{\frac{1}{2}x^1\ldots x^n\otimes [\varepsilon_1\ldots\varepsilon_m,u]-\big((sx)(x^1\ldots x^n\otimes \varepsilon_1\ldots\varepsilon_m)\big)u-(xs)(x^1\ldots x^n\otimes \varepsilon_1\ldots\varepsilon_m u)\Big\}.
\end{eqnarray*}
}}
 The second equality holds as that the weight of $\big((-sx)^k(x^1\ldots x^n\otimes \varepsilon_1\ldots\varepsilon_m)\big)_*u,k\ge 2$ is greater or equal to $n-m+1$. The third equality follows from the fact that the first term of Equation \eqref{Gutt-representation} reduces the weight by 1, that the second term of Equation \eqref{Gutt-representation} preserves the weight, and the last term of Equation \eqref{Gutt-representation} increases the weight by at least 1. The last  equality holds since $xs$ increases the weight by at least 1.
\end{proof}

On the other hand, since $sx$ increase the weight by greater than or equal to 1, thus the lowest weight term in $r \widetilde{\huaB}$ is $r (x^1\ldots x^n\otimes \varepsilon_1\ldots\varepsilon_m)$. Therefore, the constant $r$ is computed from the following equation:
\begin{eqnarray*}
&&t\Big\{\frac{1}{2}x^1\ldots x^n\otimes [\varepsilon_1\ldots\varepsilon_m,u]-\big((sx)(x^1\ldots x^n\otimes \varepsilon_1\ldots\varepsilon_m)\big)u-(xs)(x^1\ldots x^n\otimes \varepsilon_1\ldots\varepsilon_m u)\Big\}\\
&&=r (x^1\ldots x^n\otimes \varepsilon_1\ldots\varepsilon_m).
\end{eqnarray*}
Since the number $r$ is related to the notion of supertrace, we recall the definition of the latter.  Let $V$ be a finite-dimensional graded vector space; then the endomorphism algebra $\End(V)$ has a trace (or supertrace) map $\str:\End(V)\to\bk$ defined as the composition of the isomorphism $\End(V)\to V^*\otimes V$ and the evaluation map $V^*\otimes V\to \bk$. This is a standard generalization of the ordinary ungraded trace of a linear operator; if $\{e_i\}, i=1,\ldots, n$ is a basis of the even graded component of $V$, $\{e_k\}, k=n+1,\ldots n+m$ is a basis of its odd graded component and $(a_{ij})$ is the matrix of $f$ in the basis $\{e_k, k=1,\ldots n+m\}$, then
\begin{eqnarray}\label{super-tr}
\str(f)=\sum_{i=1}^na_{ii}-\sum_{i=n+1}^{n+m}a_{ii}.
\end{eqnarray}

Without loss of generality, we can set $u=e_1$ and $u=\varepsilon_1$ respectively.

\begin{lem}\label{even-map}
For $u=e_1$, we have
{\small{
\begin{eqnarray*}
&t\Big\{\frac{1}{2}x^1\ldots x^n\otimes [\varepsilon_1\ldots\varepsilon_m,u]-\big((sx)(x^1\ldots x^n\otimes \varepsilon_1\ldots\varepsilon_m)\big)u-(xs)(x^1\ldots x^n\otimes \varepsilon_1\ldots\varepsilon_m u)\Big\}\\
&=\str(\ad_{e_1})(x^1\ldots x^n\otimes \varepsilon_1\ldots\varepsilon_m).
\end{eqnarray*}
}}
\end{lem}
\begin{proof}
Since $t$ is the operator of projection onto the one-dimensional vector space $\bk\{x^1\ldots x^n\otimes \varepsilon_1\ldots\varepsilon_m\}$, one can deduce that
{\small{
\begin{eqnarray*}
&&t\Big\{\frac{1}{2}x^1\ldots x^n\otimes [\varepsilon_1\ldots\varepsilon_m,e_1]-\big((sx)(x^1\ldots x^n\otimes \varepsilon_1\ldots\varepsilon_m)\big)e_1-(xs)(x^1\ldots x^n\otimes \varepsilon_1\ldots\varepsilon_m e_1)\Big\}\\
&=&t\Big(\frac{1}{2}x^1\ldots x^n\otimes [\varepsilon_1\ldots\varepsilon_m,e_1]-x(x^2\ldots x^n\otimes \varepsilon_1\ldots\varepsilon_m)\Big)\\
&=&t\Big(\frac{1}{2}x^1\ldots x^n\otimes [\varepsilon_1\ldots\varepsilon_m,e_1]-\sum_{k=2}^{n}(-1)^{k-2}x^2\ldots d_{\CE}(x^k)\ldots x^n\otimes \varepsilon_1\ldots\varepsilon_m\\
&&-\frac{1}{2}\sum_{i=1}^{n}x^ix^2\ldots x^n\otimes [e_i,\varepsilon_1\ldots\varepsilon_m]-\frac{1}{2}\sum_{i=1}^{m}(-1)^{n-1}y^ix^2\ldots x^n\otimes [\varepsilon_i,\varepsilon_1\ldots\varepsilon_m]\Big)\\
&=&t\Big(\frac{1}{2}x^1\ldots x^n\otimes [\varepsilon_1\ldots\varepsilon_m,e_1]-\sum_{k=2}^{n}(-1)^{k-2}x^2\ldots d_{\CE}(x^k)\ldots x^n\otimes \varepsilon_1\ldots\varepsilon_m\\
&&-\frac{1}{2}x^1x^2\ldots x^n\otimes [e_1,\varepsilon_1\ldots\varepsilon_m]\Big)\\
&\stackrel{\eqref{DCE-diff}}{=}&t\big(-x^1\ldots x^n\otimes [e_1,\varepsilon_1\ldots\varepsilon_m]\big)\\
&&-t\Big(\sum_{k=2}^{n}(-1)^{k-2}x^2\ldots x^{k-1}\big(\sum_{1\le p,q\le n} -\frac{1}{2}a_{pq}^k x^p x^q+\sum_{1\le p,q\le m} -\frac{1}{2}c_{pq}^k y^p y^q\big)x^{k+1}\ldots x^n\otimes \varepsilon_1\ldots\varepsilon_m\Big)\\
&=&t\Big(-\sum_{k=1}^{m}x^1\ldots x^n\otimes \varepsilon_1\ldots [e_1,\varepsilon_k]\ldots\varepsilon_m\Big)\\
&&-t\Big(\sum_{k=2}^{n}(-1)^{k-2}x^2\ldots x^{k-1}\big(-\frac{1}{2}a_{1 k}^k x^1 x^k-\frac{1}{2}a_{k1}^k x^k x^1\big)x^{k+1}\ldots x^n\otimes \varepsilon_1\ldots\varepsilon_m\Big)\\
&=&t\Big(-\sum_{k=1}^{m}x^1\ldots x^n\otimes \varepsilon_1\ldots \varepsilon_{k-1}(b_{1k}^k\varepsilon_k)\varepsilon_{k+1}\ldots\varepsilon_m\Big)\\
&&-t\Big(\sum_{k=2}^{n}x^1x^2\ldots x^{k-1}(-a_{1 k}^kx^k)x^{k+1}\ldots x^n\otimes \varepsilon_1\ldots\varepsilon_m\Big)\\
&=&\Big(\sum_{k=1}^{n}a_{1k}^k-\sum_{k=1}^{m}b_{1k}^k\Big)(x^1\ldots x^n\otimes \varepsilon_1\ldots\varepsilon_m)\\
&\stackrel{\eqref{super-tr}}{=}&\str(\ad_{e_1})(x^1\ldots x^n\otimes \varepsilon_1\ldots\varepsilon_m).
\end{eqnarray*}
}}
The first equality follows from the identity $t\Big(\big((sx)(x^1\ldots x^n\otimes \varepsilon_1\ldots\varepsilon_m)\big)e_1\Big)=0$ and the definition of $s$. The second equation follows since the operator $x$ of Equation \eqref{differential-x} increases the weight by at least 1.  The third  equality follows from the identities $x^ix^i=0$ and $t(y^jx^2\ldots x^n\otimes [\varepsilon_j,\varepsilon_1\ldots\varepsilon_m])=0,j=1,\cdots,m$. The fourth  equality is obtained by the concrete formula of the CE differential \eqref{DCE-diff}. The fifth equality holds since $x^ix^i=0$ and $t(x^2\ldots x^{k-1}y^p y^qx^{k+1}\ldots x^n\otimes \varepsilon_1\ldots\varepsilon_m)=0$. The last three equalities are obtained by direct computation.
\end{proof}

On the other hand, we have the following result.
\begin{lem}\label{odd-map}
For $u=\varepsilon_1$, we have
{\small{
\begin{eqnarray*}
t\Big\{\frac{1}{2}x^1\ldots x^n\otimes [\varepsilon_1\ldots\varepsilon_m,\varepsilon_1]-&\big((sx)(x^1\ldots x^n\otimes \varepsilon_1\ldots\varepsilon_m)\big)\varepsilon_1-(xs)(x^1\ldots x^n\otimes \varepsilon_1\ldots\varepsilon_m \varepsilon_1)\Big\}\\=&\str(\ad_{\varepsilon_1})(x^1\ldots x^n\otimes \varepsilon_1\ldots\varepsilon_m)\\=&0.
\end{eqnarray*}
}}
\end{lem}

\begin{proof}The operator $\ad_{\varepsilon_1}$ is odd and so, it has vanishing (super)trace. Next,
by the definitions of $s$ and $x$, we have
{\small{
\begin{eqnarray*}
&&t\Big\{\frac{1}{2}x^1\ldots x^n\otimes [\varepsilon_1\ldots\varepsilon_m,\varepsilon_1]-\big((sx)(x^1\ldots x^n\otimes \varepsilon_1\ldots\varepsilon_m)\big)\varepsilon_1-(xs)(x^1\ldots x^n\otimes \varepsilon_1\ldots\varepsilon_m \varepsilon_1)\Big\}\\
&=&t\Big\{\sum_{k=1}^{m}\frac{(-1)^{m-k}}{2}x^1\ldots x^n\otimes \varepsilon_1\ldots [\varepsilon_k,\varepsilon_1]\ldots\varepsilon_m\Big\}-t\Big\{s\Big(d_{\CE}(x^1\ldots x^n)\otimes \varepsilon_1\ldots\varepsilon_m\\
&&+\frac{1}{2}\sum_{i=1}^{n}x^ix^1\ldots x^n\otimes [e_i,\varepsilon_1\ldots\varepsilon_m]+\frac{1}{2}\sum_{i=1}^{m}(-1)^{n}y^ix^1\ldots x^n\otimes [\varepsilon_i,\varepsilon_1\ldots\varepsilon_m]\Big)\varepsilon_1\Big\}\\
&=&-t\Big\{s\Big(\sum_{k=1}^{n}(-1)^{k-1}x^1\ldots d_{\CE}(x^k)\ldots x^n\otimes \varepsilon_1\ldots\varepsilon_m\Big)\varepsilon_1\Big\}\\
&&-t\Big\{s\Big(\frac{1}{2}\sum_{i=1}^{m}\sum_{s=1}^{m}(-1)^{n+s-1}y^ix^1\ldots x^n\otimes \varepsilon_1\ldots[\varepsilon_i,\varepsilon_s]\ldots\varepsilon_m\Big)\varepsilon_1\Big\}\\
&\stackrel{\eqref{DCE-diff}}{=}&-t\Big\{s\Big(\sum_{k=1}^{n}(-1)^{k-1}x^1\ldots x^{k-1}\big(\sum_{1\le p,q\le n} -\frac{1}{2}a_{pq}^k x^p x^q+\sum_{1\le p,q\le m} -\frac{1}{2}c_{pq}^k y^p y^q\big)x^{k+1}\ldots x^n\otimes \varepsilon_1\ldots\varepsilon_m\Big)\varepsilon_1\Big\}\\
&&-t\Big\{s\Big(\frac{1}{2}\sum_{i=1}^{m}\sum_{s=1}^{m}(-1)^{n+s-1}y^ix^1\ldots x^n\otimes \varepsilon_1\ldots\varepsilon_{s-1}\big(\sum_{l=1}^{n}c_{is}^{l}e_l\big)\varepsilon_{s+1}\ldots\varepsilon_m\Big)\varepsilon_1\Big\}\\
&=&0.
\end{eqnarray*}
}}
The first equality follows since the operator $x$ in Equation \eqref{differential-x} increases the weight by at least 1,  $s$ preserves the weight and $\varepsilon_1\varepsilon_1=0$.
The second equality holds since $t(x^1\ldots x^n\otimes \varepsilon_1\ldots [\varepsilon_k,\varepsilon_1]\ldots\varepsilon_m)=0$ and $x^ix^i=0$.
 The third equation is obtained from the concrete formula of the CE differential \eqref{DCE-diff} and the structure constants \eqref{structural constant-1} of the graded Lie algebra $\g$. The last equality is  obtained by the direct computation via the definitions of $s$ and $t$.
\end{proof}

Combining  Lemma \ref{even-map} and Lemma \ref{odd-map}, we have the following theorem.
\begin{theorem}\label{thm:super-g-trace}
The right dg  $\U(\g)$-module  $\CE^\bullet(\g,\U(\g))$ is quasi-isomorphic to the  one dimensional right  dg $\U(\g)$-module $\bk[-|\g|]$ (with vanishing differential). The right $\U(\g)$-module structure on $\bk[-|\g|]$  is given by the supertrace
\begin{eqnarray}\label{dg-U-mod-2}
	k.u=\str(\ad_{u})k,~~u\in \g,k\in\bk.
\end{eqnarray}
If $x^1,\ldots, x^n$ is the odd part of a basis of $(\g[1])^*$ (corresponding to the even part of the corresponding basis of $\g$) and $\varepsilon_1,\ldots, \varepsilon_m$ is the odd part of a basis of $\g$, a cocycle representing the only (up to a scalar multiple) nontrivial cohomology class of $\CE^\bullet(\g,\U(\g))$ can be chosen to be the deformed Berezinian
\[
\widetilde{\huaB}:=\alpha(x^1\ldots x^n\otimes \varepsilon_1\ldots\varepsilon_m)
\]
where $\alpha$ is defined in Equation (\ref{eq:alpha}).
\end{theorem}

Recall from \cite{BL} that a graded Lie algebra $\g$ is \emph{unimodular} if for any $u\in\g$ the operator $\ad_u$ has vanishing supertrace. We have the following corollary.
\begin{cor}
	Let $\g$ be a unimodular graded Lie algebra. Then its universal enveloping algebra $\U(\g)$ is a graded $|\g|$-Gorenstein algebra. The pair $\big(\Perf(\U(\g)), \operatorname{pvd}(\U(\g))\big)$ is an open $|\g|$-CY structure on the category $D(\U(\g))$.
\end{cor}
\begin{proof}
	This follows immediately from Theorem \ref{thm:super-g-trace} and Theorem \ref{cor:HopfCY}.
\end{proof}
Let us now discuss MC twisting of dg comodules over dg coalgebras (obtained by dualization from MC twisting of the corresponding pc dg modules over pc dg algebras).
\begin{prop}\label{Twisted-theory}
Let $C$ be a dg coalgebra and  $\xi\in C^*$ be an MC element of the dual pc algebra $C^* $. Define $d^\xi:C\to C$ as following:
\begin{eqnarray}\label{twisted-coalgebra}
d^{\xi}(c)=d(c)-\sum_{c}\xi(c_{(1)})c_{(2)}+\sum_{c}(-1)^{|c_{(1)}|}\xi(c_{(2)})c_{(1)}, \forall c\in C.
\end{eqnarray}
Here we use Sweedler's notation of coalgebras: $\Delta(c)=\sum_{c}c_{(1)}\otimes c_{(2)}\in C\otimes C,\forall c\in C$.
Then the  graded coalgebra $C$ equipped with $d^\xi$ is a dg coalgebra, which is denoted by $C^\xi$. Let $V$ be a right dg  $C$-comodule. Define $d^{[\xi]}:V\to V$ as following:
\begin{eqnarray}\label{twisted-comod}
d^{[\xi]}(v)=d(v)+\sum_{v}(-1)^{|v_{(0)}|}\xi(v_{(1)})v_{(0)}, \forall v\in V.
\end{eqnarray}
Here we use Sweedler's notation of comodules: $\Delta(v)=\sum_{v}v_{(0)}\otimes v_{(1)}\in V\otimes C,\forall v\in V$.
Then the  graded comodule $V$ equipped with $d^{[\xi]}$ is a right dg  $C^\xi$-comodule, which is denoted by $V^{[\xi]}$.
\end{prop}

\begin{proof}
Since $\xi$ is a MC element of $C^*$, one can obtain the twisted algebra $(C^*)^{\xi}$ which is the dg algebra with the same underlying graded algebra as $C^*$ and differential
$$
d^{\xi}(\alpha)=d(\alpha)+[\xi,\alpha],~\forall\alpha\in C^*.
$$
Consider the continuous linear duality of $(C^*)^{\xi}$, we deduce that $[(C^*)^{\xi}]^*=C$ is a dg coalgebra with the differential \eqref{twisted-coalgebra}.

Let $V$ be a right dg  $C$-comodule. Then $V^*$ is a right pc  dg  $C^*$-module. Moreover, we have the twisted module $(V^*)^{[\xi]}$ which is the right dg
$(C^*)^{\xi}$-module with the same underlying graded module as $V^*$ and differential
$$
d^{[\xi]}(f)=d(f)-(-1)^{|f|}f.\xi,~\forall f\in V^*.
$$
We deduce that $[(V^*)^{[\xi]}]^*\cong V$, the continuous linear dual to $(V^*)^{[\xi]}$  is a right dg  $C^\xi$-comodule with the differential \eqref{twisted-comod} as claimed.
\end{proof}

\emptycomment{
\begin{theorem}\label{derived-CE-comod}
Let $\g$ be a finite dimensional $\mathbb Z$-graded Lie algebra over the field $\bk$ and $M$ be a finite dimensional $\mathbb Z$-graded $\g$-module. Then the derived linear dual $d(\CE_\bullet(\g, M))$ of the dg comodule $\CE_\bullet(\g,M)$ is weakly equivalent to
the right dg  $\CE_\bullet(\g)$-comodule $\CE_\bullet(\g)[|\g|]$ has $\CE_\bullet(\g)[|\g|]$ as its comodule structure  and the codifferential $\widetilde{d_{\CE}}$ is determined by the formula
\begin{eqnarray*}
\widetilde{d_{\CE}}(u_1[1]\ldots u_k[1])
&=&(-1)^{|\g|}
\Big(d_{\CE}(u_1[1]\ldots u_k[1])\\&+&\sum_{i=1}^{k}(-1)^{|u_i[1]|(|u_1[1]|+\ldots+|u_{k-1}[1]|)}\str(\ad_{u_i})u_1[1]\ldots \widehat{u_i[1]}\ldots u_k[1]\Big).
\end{eqnarray*}
\end{theorem}
}
\begin{rem}
	If $C$ is a cocommutative dg coalgebra (such as $\CE_\bullet(\g)$ for a dg Lie algebra $\g$) then $C^\xi\cong C$ and the categories of left and right dg $C$-comodules are isomorphic.
\end{rem}

\begin{theorem}\label{derived-CE-comod}
Let $\g$ be a finite dimensional $\mathbb Z$-graded Lie algebra over the field $\bk$ and $M$ be a finite dimensional $\mathbb Z$-graded $\g$-module. Then the derived linear dual $d(\CE_\bullet(\g, M))$ of the dg comodule $\CE_\bullet(\g,M)$ is weakly equivalent to
the right dg  $\CE_\bullet(\g)$-comodule $\CE_\bullet\big(\g, (M^*[-|\g|])^{\tw}\big)$.  The underlying graded vector space of  the right $\U(\g)$-module $(M^*[-|\g|])^{\tw}$ is $M^*[-|\g|]$ and the right $\U(\g)$-module structure  is given by the formula
\begin{eqnarray}\label{DG-U-mod}
	\nu[-|\g|]._{\tw}u=(\nu.u)[-|\g|]+\str(\ad_{u})\nu[-|\g|],~~u\in \g,\nu\in M^*.
\end{eqnarray}
Here right action of $\U(\g)$ on $M^*$ corresponds to its left action via the antipode in $\U(\g)$.
\end{theorem}

\begin{proof}
By Proposition \ref{D-dual-CE}, we deduce that $d(\CE_\bullet(\g, M))$ is weakly equivalent to $\CE_\bullet(\g, \bd\otimes M^*)$. Since the dualizing complex $\bd=\CE^\bullet(\g, \U(\g))$ of a finite dimensional $\mathbb Z$-graded Lie algebra $\g$ is quasi-isomorphic to the  right dg $\U(\g)$-module $\bk[-|\g|]$ (with vanishing differential) and the right $\U(\g)$-module structure on $\bk[-|\g|]$  is given by the supertrace, it follows that the right $\U(\g)$-module structure  on $\bd\otimes M^*$ is quasi-isomorphic to $M^*[-|\g|]$ and the right $\U(\g)$-module structure  is given by \eqref{DG-U-mod} as claimed. Thus we deduce that $d(\CE_\bullet(\g, M))$ is weakly equivalent to
the right dg comodule $\CE_\bullet\big(\g, (M^*[-|\g|])^{\tw}\big)$.
\end{proof}
Let $W$ be a graded vector space  of finite total dimension. Let us choose a basis $x_1,\ldots,x_n$ of $W^*$.  A $\bk$-linear derivation (or a formal vector field on $W$) $\xi$ of the algebra $\hat{\huaS}(W)$ is determined by its values on the generators $x_i, i=1,\ldots, n$ and can therefore be written as a sum
\[
\xi=\sum_{i=1}^n f_i\frac{\partial}{\partial x_i}
\]
where $f_i\in \hat{\huaS}(W)$. Here the degree of the vector field $\xi$ is $|f_i|-|x_i|$.

\begin{defi}
	The \emph{divergence} of a formal vector field $\xi$ on $W$ is defined as
	\[\nabla(\xi)=\sum_{i=1}^n (-1)^{|f_i||x_i|}\frac{\partial f_i}{\partial x_i}.\]
	One has $|\nabla(\xi)|=|\xi|$.
	It is easy to see that   $\nabla(\xi)$ is defined independently of the basis in $A$ and that it obeys the following well-known rules for any vector field $\xi$ and $f\in \hat{\huaS}(W)$:
	\begin{equation}\label{eq:divergenceproperty}
		\nabla(f\xi)=f\nabla\xi+(-1)^{|f||\xi|}\xi(f).
	\end{equation}
\end{defi}
The following result holds.

\emptycomment{
\begin{cor}\label{cor:dualcomodule}
Let $\g$ be a finite dimensional $\mathbb Z$-graded Lie algebra over the field $\bk$. Then the derived linear dual $d(\CE_\bullet(\g))$ of the dg comodule $\CE_\bullet(\g)$ is weakly equivalent to the right dg  $\CE_\bullet(\g)$-comodule $\CE_\bullet(\g)^{[\nabla(d_{\CE})]}[|\g|]$.
\end{cor}
}

\begin{cor}\label{cor:dualcomodule}
Let $\g$ be a finite dimensional $\mathbb Z$-graded Lie algebra over the field $\bk$ and $M$ be a finite dimensional $\mathbb Z$-graded $\g$-module. Then the derived linear dual $d(\CE_\bullet(\g, M))$ of the dg comodule $\CE_\bullet(\g,M)$ is weakly equivalent to
the right dg  $\CE_\bullet(\g)$-comodule $\CE_\bullet(\g,M^*)^{[\nabla(d_{\CE})]}[-|\g|]$.
\end{cor}

\begin{proof}
We assume that $\{e_1,\ldots,e_n;\varepsilon_{1},\ldots,\varepsilon_{m}\}$ is a  homogeneous basis of $\g$ and with elements $\{e_1,\ldots,e_n\}$ even and $\{\varepsilon_{1},\ldots,\varepsilon_{m}\}$  odd. Set $x_i:=e_i[1]$ by $x_i$ and $y_i:=\varepsilon_i[1]$. Consider the basis $\{x^1,\ldots,x^n;y^1,\ldots,y^m\}$ of  $(\g[1])^*$ dual to $\{x_1,\ldots,x_n;y_1,\ldots,y_m\}$.
Let $d_1$ and $d_2$ be two formal vector fields on $\g[1]$. We have the following equation
\begin{eqnarray*}
\nabla[d_1,d_2]=d_1(\nabla d_2)-(-1)^{|d_1||d_2|}d_2(\nabla d_1).
\end{eqnarray*}
Thus, we deduce that $\nabla[d_{\CE},d_{\CE}]=2d_{\CE}(\nabla(d_{\CE}))=0$. It follows $\nabla(d_{\CE})$ is an MC element of the CE algebra $\CE^\bullet(\g)$.
By a direct calculation, we deduce that
\begin{eqnarray*}
[\nabla(d_{\CE})]&=&\sum_{q=1}^{n} \left\{\sum_{k=1}^{n}a_{qk}^k-\sum_{k=1}^{m}b_{qk}^k \right\}x^q\\
                 &=&\sum_{q=1}^{n} \str(\ad_{e_q})x^q.
\end{eqnarray*}
By Proposition \ref{Twisted-theory}, we deduce that $\CE_\bullet(\g, M^*)^{[\nabla(d_{\CE})]}$ is a dg right $\CE_\bullet(\g)$-comodule and the codifferential $d^{[\nabla(d_{\CE})]}$ is
{\small{
\begin{eqnarray*}
&&d^{[\nabla(d_{\CE})]}(\nu\otimes u_1[1]\ldots u_k[1])\\
&\stackrel{\eqref{twisted-comod}}{=}& d_{\CE}(\nu\otimes u_1[1]\ldots u_k[1])+\sum_{i=1}^{k}(-1)^{|\nu|+|u_1[1]|+\ldots+|u_{i-1}[1]|}[\nabla(d_{\CE})](u_i[1]) \nu\otimes u_1[1]\ldots \widehat{u_i[1]}\ldots u_k[1]\\
&=&d_{\CE}(\nu\otimes u_1[1]\ldots u_k[1])+{\sum_{i=1}^{k}(-1)^{|\nu|+|u_1[1]|+\ldots+|u_{i-1}[1]|}\str(\ad_{u_i}) \nu\otimes u_1[1]\ldots \widehat{u_i[1]}\ldots u_k[1]}.
\end{eqnarray*}
}}
Here $\nu$ is an element of $M^*$ and $u_1[1],\ldots,u_k[1]$ are elements of $\g[1]$.

On the other hand, the twisted codifferential $\widetilde{d_{\CE}}$ on $\CE_\bullet\big(\g, (M^*[-|\g|])^{\tw}\big)$ is determined by the formula
{\small{
\begin{eqnarray*}
&&\widetilde{d_{\CE}}(\nu[-|\g|]\otimes u_1[1]\ldots u_k[1])\\
&=&(-1)^{|\nu|+|\g|}
\nu[-|\g|]\otimes d_{\CE}(u_1[1]\ldots u_k[1])\\&+&\sum_{i=1}^{k}(-1)^{\Theta_i}(-1)^{|\nu|+|\g|}\Big(\big((\nu.u_i)[-|\g|]+\str(\ad_{u_i})\nu[-|\g|]\big)\otimes u_1[1]\ldots \widehat{u_i[1]}\ldots u_k[1]\Big),
\end{eqnarray*}
}}
here $\Theta_i=|u_i[1]|(|u_1[1]|+\ldots+|u_{i-1}[1]|)$. Thus, we deduce that $\widetilde{d_{\CE}}=(-1)^{|\g|}d^{[\nabla(d_{\CE})]}$.
By Theorem \ref{derived-CE-comod}, it follows that $d(\CE_\bullet(\g, M))$ is  weakly equivalent to the right dg  comodule $\CE_\bullet(\g,M^*)^{[\nabla(d_{\CE})]}[-|\g|]$. This finishes the proof.
\end{proof}

Now let $\g$ be an ordinary (ungraded) finite-dimensional Lie algebra and $M$ be a finite-dimensional module. In that case $\CE_\bullet(\g,M)$ is finite-dimensional $\CE_\bullet(\g)$-comodule. Since the derived linear dual of a finite dimensional comodule is the ordinary linear dual, we recover Hazewinkel's result \cite{Haz}:
\begin{cor}\label{cor:Haz}
Let $\g$ be an $n$-dimensional Lie algebra over the field $\bk$ and $M$ be a finite dimensional left $\g$-module. Then one deduces $$\HCE_{i}\big(\g,(M^*)^\tw\big)\cong\HCE^{n-i}(\g,M^*),~~i=0,1,\ldots,n.$$
The underlying vector space of  the right $\g$-module $(M^*)^{\tw}$ is $M^*$ and the right $\g$-module structure  is given by the formula
\begin{eqnarray}\label{DG-U-mod-1}
	\nu._{\tw}u=(\nu.u)+\Str(\ad_{u})\nu,~~u\in \g,\nu\in M^*.
\end{eqnarray}
Here the right action $\nu.u$, as before, corresponds to the left action of $\U(\g)$ on $M^*$, derived from the  $\g$-module structure on $M$.
\end{cor}
\subsection{Dg Lie algebras revisited} Let us now assume that $\g$ is a proper dg algebra. Recall from Theorem \ref{the:1-dim} that the dualizing complex $\CE^\bullet(\g,\U(\g))$ has 1-dimensional cohomology. It is also a right dg $\U(\g)$-module. Therefore, the one-dimensional space $\HCE^ \bullet(\g,\U(\g))$ is a right $\HH^\bullet(\g)$-module. The following result holds.
\begin{prop}\label{prop:action}
	The right $\HH^\bullet(\g)$-module structure on $\bk[-|\HH^\bullet(\g)|]\cong\HCE^ \bullet(\HH^\bullet(\g),\U(\HH^\bullet(\g)))$ is given by the supertrace of the adjoint representation of the graded Lie algebra $\HH^\bullet(\g)$: an element $a\in \HH^\bullet(\g)$ acts on $\bk[-|\HH^\bullet(\g)|]$ as the multiplication by $\str\ad_a$.
\end{prop}

\begin{proof}
	Recall that the spectral sequence for $\HCE^\bullet(\g,\U(\g))$ has its $E_2$-term isomorphic to $\HCE^\bullet(\HH^\bullet(\g),\U(\HH^\bullet(\g))$, and it is one-dimensional, forcing the collapse of the spectral sequence. Since $\HCE^ \bullet(\HH^\bullet(\g),\U(\HH^\bullet(\g)))$ is one-dimensional,  any element in the graded Lie algebra $\HH^\bullet(\g)$ acts on it  as a multiplication by a scalar, and it follows from Theorem \ref{thm:super-g-trace} that this action modulo higher filtration terms is the supertrace of the adjoint representation. Again, using one-dimensionality of $\HCE^ \bullet(\HH^\bullet(\g),\U(\HH^\bullet(\g)))$, we see that no higher correction terms to the above action are present.
\end{proof}
Let us now assume that $\g$ is a \emph{connective} dg Lie algebra, i.e. $\g$ is concentrated in cohomological degrees $\leq 0$. In that case the zero-truncation gives a map of dg Lie algebras $\g\to\HH^0(\g)$. Then we have the following result.
\begin{cor}\label{cor:connective}
	If $\g$ is a proper connective dg Lie algebra, then the dualizing module $\CE^\bullet(\g,\U(\g))$ is quasi-isomorphic to the one-dimensional space $\bk[-|\HH^\bullet(\g)|]$ with the action of $\g$ given by the composition
	\[\xymatrix{
	\g\ar[r]&\HH^0(\g)\ar^{\str}[r]&\bk
.}
	\]
\end{cor}
\begin{proof}
	By Proposition \ref{prop:action}, the action of $\HH^\bullet(\g)$ (or, equivalently, of $\U(\HH^\bullet(\g))$) is through the map
		\[\xymatrix{
		\HH^\bullet(\g)\ar[r]&\HH^0(\g)\ar^{\str}[r]&\bk,
	}
	\]
	where the first map above is just the projection onto the zeroth graded component. But the grading assumption on $\HH^\bullet(\g)$ (and, thus, on $\U(\HH^\bullet(\g))$) imply that the higher $A_\infty$-components of the action vanish, proving the claim.
\end{proof}
Our next application is related to rational homotopy theory, so we set $\bk:=\mathbb{Q}$. Let $X$ be a connected rationally nilpotent topological space of finite type and $A(X)$ be its minimal Sullivan-de Rham model. It is an augmented commutative dg algebra determined up to quasi-isomorphism and faithfully recording the rational homotopy type of $X$. Then the dual bar-construction $(\B(A))^*$ is a dg Hopf algebra and is a universal enveloping algebra of a dg Lie algebra $\g:=\g(A)$ called the Lie-Quillen model of $X$, see \cite{nei} for this approach to rational homotopy theory. The cohomology $\HH^\bullet(\g)$ is, up to regrading, the rational homotopy groups of $X$, except $\HH^0(\g)$ is the rational Malcev Lie algebra  of $\pi_1(X)$.

\begin{cor}\label{cor:LieQuillen}
	Let $\g$ be a Lie-Quillen model of a nilpotent rational space of finite type and having finite rational Postnikov tower. Then $\U(\g)$ is a Gorenstein dg algebra and so, $D(\U(\g))$ has an open CY structure given by a pair $\big(\Perf(\U(\g)),\operatorname{pvd}(\U(\g))\big)$.
\end{cor}	
\begin{proof}
	The assumptions on $\g$ (or, rather on the space of which it is a Lie-Quillen model) ensure that $\g$ is a proper dg Lie algebra. Therefore,  by Corollary \ref{cor:connective}, the dg Lie algebra $\g$ acts on its dualizing complex through the supertrace of the adjoint action of $\HH^\bullet(\g)$ but this action is nilpotent and so, its (super)trace is trivial. By Corollary \ref{cor:HopfCY}, $D(\U(\g))$ has an open CY structure.
\end{proof}
Recall that associated to a connected pointed topological space $X$, is the triangulated category $\operatorname{Loc}_{\infty}(X)$ of infinity-local systems. The category $\operatorname{Loc}_{\infty}(X)$ is an important homotopical invariant of $X$, and it is equivalent to the derived category of cohomologically locally constant sheaves on $X$ as well as $D(C_\bullet L(X))$, the derived category of the singular chain algebra on $L(X)$, the based loop space on $X$, cf. for example \cite[Section 4.3]{HCL} and references therein. We will consider here $\mathbb{Q}$-linear infinity local systems (which means, in particular, that the singular chain algebra on the loop space of $X$ is taken with coefficients in $\mathbb{Q}$) and assume, in addition, that $X$ is simply-connected. In that case the dg algebra $C_\bullet L(X)$ is quasi-isomorphic to $\U(\g)$ where $\g$ is the Lie-Quillen model of $X$,\cite[Theorem 26.5]{FHT}. Then we have the following result.
\begin{cor}\label{cor:simply-connected}
	Let $\bk=\mathbb{Q}$. Given a rational simply-connected topological space $X$ of finite type and with finite rational Postnikov tower, its category of $\infty$-local systems $\operatorname{Loc}_{\infty}(X)$ has the structure of an open CY-category given by a pair $\Big(\Perf\big(C_\bullet L(X)\big),\operatorname{pvd}\big(C_\bullet L(X)\big)\Big)$.
\end{cor}
\begin{proof}
	This follows at once from Corollary \ref{cor:LieQuillen}.
\end{proof}
\begin{rem}
	A simply-connected space is \emph{elliptic} if it has totally finite-dimensional rational cohomology and rational homotopy groups, cf. \cite{FHT} regarding this notion. It is known that elliptic spaces have Poincar\'e duality in its rational cohomology. Spaces considered in Corollary \ref{cor:simply-connected} are more general than elliptic spaces in that total finite-dimensionality of its rational cohomology is not assumed (though total finite-dimensionality of its rational homotopy is). It may be natural to call simply-connected topological spaces $X$ for which $\pi_\bullet(X,\mathbb{Q})$ is totally finite dimensional $\mathbb{Q}$-vector space, \emph{generalized elliptic spaces}. The simplest example of a generalized elliptic space that is not elliptic is $\operatorname{CP}^\infty$, the infinite complex projective space.  We see that the category of infinity local systems on generalized elliptic spaces has a CY structure (which boils down to ordinary Poincar\'e duality for genuinely elliptic spaces). In fact, we see that a generalized elliptic space $X$ is a rational \emph{Frobenius space} in the sense of \cite{BCL}: the dg algebra of chains on based loops of $X$ is Gorenstein.
\end{rem}
\subsection{Dualizing complexes for Lie superalgebras}
In this subsection, we assume that $\g=\g_0\oplus\g_1$ is a finite dimensional Lie superalgebra (i.e. $\mathbb{Z}/2$-graded Lie algebra) over the field $\bk$ of (super)dimension $n|m$ (i.e. $\dim \g_0=n$ and $\dim \g_1=m$). In this situation, all results obtained in the $\mathbb{Z}$-graded case, continue to be valid, with the same proofs. We will limit ourselves with formulating an analogue of Theorem \ref{thm:super-g-trace}. As before, let $\{e_1,\ldots,e_n\}$ be a basis of $\g_0$ and $\{\varepsilon_{1},\ldots,\varepsilon_{m}\}$ be a basis of $\g_1$. We denote $e_i[1]$ by $x_i$ and $\varepsilon_i[1]$ by $y_i$ respectively. Then  $\{x_1,\ldots,x_n;y_1,\ldots,y_m\}$ is a basis of $\g[1]$ and $\{x^1,\ldots,x^n;y^1,\ldots,y^m\}$ is the dual basis of $(\g[1])^*$.
By a  calculation, similar to that of $\mathbb Z$-graded Lie algebras, we deduce the following result.
\begin{theorem}\label{thm:supertrace}
	The right dg  $\U(\g)$-module  $\CE^\bullet(\g,\U(\g))$ is quasi-isomorphic to the  one dimensional right dg $\U(\g)$-module $\bk$ placed in degree $n+m\mod 2$ (with vanishing differential). The right $\U(\g)$-module structure on $\bk$  is given by the supertrace
	\begin{eqnarray}\label{dg-U-mod-3}
		k.u=\str(\ad_{u})k,~~u\in \g,k\in\bk.
	\end{eqnarray}
	If $x^1,\ldots, x^n$ is the odd part of a basis of $(\g[1])^*$ (corresponding to the even part of the corresponding basis of $\g$) and $\varepsilon_1,\ldots, \varepsilon_m$ is the odd part of a basis of $\g$, a cocycle representing the only (up to a scalar multiple) nontrivial cohomology class of $\CE^\bullet(\g,\U(\g))$ can be chosen to be the deformed Berezinian
	\[
	\widetilde{\huaB}:=\alpha(x^1\ldots x^n\otimes \varepsilon_1\ldots\varepsilon_m)
	\]
	where $\alpha$ is defined in Equation $(\ref{eq:alpha})$.
\end{theorem}
\subsection{Dualizing complexes for $L_\infty$-algebras} Recall that an $L_\infty$-algebra structure on a graded vector space $\g$ is a continuous derivation $\ell$ of the graded pc algebra $\hat{\huaS}(\g^*[-1])$ having no constant term and such that $[\ell,\ell]=0$. For simplicity, we will assume that $\dim(\g)<\infty$. Such a derivation is determined by its restriction on $\g^*[-1]$ which is a linear map $$\g^*[-1]\to \hat{\huaS}(\g^*[-1])\cong \prod_{k=0}^\infty \huaS^k(\g^*[-1]).$$ So, the derivation $\ell$ has components $l_k: \g^*[-1]\to \huaS^k(\g^*[-1])$ and the requirement of the vanishing constant term implies that $l_0=0$. The components $l_k$ of $\ell$ are equivalent by duality to maps $\Lambda^k\g\to\g$ which satisfy a collection of identities equivalent to the equation $[\ell,\ell]=0$. In particular, if $l_k=0$ for $k\neq 1,2$, we have a map $\g\to\g$ of degree 1 squaring to zero, and a map $\Lambda^2\g\to\g$ of degree 0 satisfying the Jacobi identity, plus a compatibility condition between the two. In other words, these data reduce to that of a dg Lie algebra (and if $l_1=0$ this further simplifies to a graded Lie algebra).

Thus, $L_\infty$-algebras generalize dg Lie algebras. An important subclass of $L_\infty$-algebras is formed by \emph{minimal} $L_\infty$-algebras, for which the component $l_1:\g^*[-1]\to \g^*[-1]$ vanishes.  We will write $\CE^\bullet(\g)$ for the  graded pc algebra $\hat{\huaS}(\g^*[-1])$ supplied with the differential $\ell$ and call it the representing pc dg algebra for the $L_\infty$-algebra $\g$. Note that this is consistent with our notations in the case when $\g$ is a dg Lie algebra.

It is natural to ask whether the results of this section hold true under the assumption that $\g$ is an $L_\infty$-algebra. A generalization of our methods to the $L_\infty$-case appears highly nontrivial, particularly because the methods we employed use in an essential way universal enveloping algebras of  graded Lie algebras.  For $L_\infty$-algebras, this notion is considerably more involved (see \cite{Bar, Mor, KH}). Furthermore, it is virtually certain that the formulas for the action of the universal enveloping algebra $\U(\g)$ of an $L_\infty$-algebra $\g$
 on the deformed Berezinian will include higher $A_\infty$-terms. Therefore,
we approach it from the Koszul dual point of view of dg modules over the pc dg algebra $\CE^\bullet(\g)$ (or, equivalently, comodules over the dual dg coalgebra $\CE_\bullet(\g)$).

The derivation $\ell$ of $\CE^\bullet(\g)$ has divergence $\nabla(\ell)$ which is a 1-cocycle in $\CE^\bullet(\g)$ (this is proved in the same way as for graded Lie algebras) and so, we can form the MC twist
$\CE_\bullet(\g)^{[\nabla(\ell)]}$ as a comodule over $\CE_\bullet(\g)$.

The following conjectural result is an $L_\infty$-version of Corollary \ref{cor:dualcomodule}.

\begin{con}\label{conjecture-comod}
	Let $(\g,\ell)$ be a finite dimensional $L_\infty$-algebra over the field $\bk$. Then the derived linear dual $d(\CE_\bullet(\g))$ of the dg  $\CE_\bullet(\g)$-comodule $\CE_\bullet(\g)$ is weakly equivalent to
the right dg  $\CE_\bullet(\g)$-comodule $\CE_\bullet(\g)^{[\nabla(\ell)]}[-|\HH^\bullet(\g)|]$.
\end{con}

\begin{rem}
	It is likely that $L_\infty$-versions of other results in this section follow from Conjecture \ref{conjecture-comod}. For example, let $\g$ be a finite-dimensional $L_\infty$-algebra and $M$ be a finite-dimensional $L_\infty$-module over $\g$, which means that $M$ is a dg $\CE_\bullet(\g)$-comodule which is isomorphic to $\CE_\bullet(\g)\otimes V$ when the differential is disregarded and where $V$ is any finite-dimensional vector space. Such an $M$ corresponds by Koszul duality to a proper dg $\Omega(\CE_\bullet(\g))$-module $M^!$. Next, $\Omega(\CE_\bullet(\g))$ is a dg Hopf algebra, and so, by Proposition \ref{prop:duality}, we have $d(M^!)\simeq \Hom(M^!,\bd)$. Since $\bk$-linear Hom  of $\Omega(\CE_\bullet(\g))$-modules corresponds to $\CE^\bullet(\g)$-linear Hom of $\CE_\bullet(\g)$-comodules, we get
	\[
	d(M)\simeq {\RHom}_{\CE^\bullet(\g)}\big(M,d(\CE_\bullet(\g))\big).
	\]
	This is an $L_\infty$-analogue of Corollary \ref{cor:dualcomodule}. Indeed, by Conjecture \ref{conjecture-comod}, $d(\CE_\bullet(\g))$ is a cofree $\CE_\bullet(\g)$-comodule (with differentials disregarded) and we conclude that $d(M)$ has underlying $\CE_\bullet(\g)$-comodule  $\CE_\bullet(\g)[-|\HH^\bullet(\g)|]\otimes V^*$; the differential in it is suitably twisted.
\end{rem}
\section{Derivations and traces on algebras of formal Laurent series}\label{sec: Laurent}
In this section we introduce certain algebras of formal Laurent series and traces (or formal integrals) on them.

Let $\bk[[y]][y^{-1}]$ be the graded field of formal Laurent series in one variable $y$ of even degree $|y|\in\mathbb Z$ (note that unless $|y|=0$, this will, in fact, be a graded field of Laurent polynomials).  It admits   an inner product of graded degree $|y|$ 
\[
\langle f,g\rangle:=\res(fg)
\]
where $f,g\in \bk[[y]][y^{-1}]$ and $\res$ is the usual residue of a Laurent power series, the coefficient at $y^{-1}$. This inner product is nondegenerate; it exhibits $\bk[[y]]\subset \bk[[y]][y^{-1}]$ as dual vector space to $y^{-1}\bk[y^{-1}]\subset \bk[[y]][y^{-1}]$  and  $y^{-1}\bk[y^{-1}]$ as the \emph{continuous dual} to $\bk[[y]]$. This pairing is also invariant: $\langle fg,h\rangle=\langle f,gh\rangle$ for $f,g,h\in \bk[[y]][y^{-1}]$ and thus, is determined by a trace, or a Berezin integral  $f\mapsto \langle f,1\rangle$ on $\bk[[y]][y^{-1}]$.

An \emph{odd} analogue of $\bk[x]$ is the algebra $\Lambda(x)$; the exterior algebra on a generator $x$ of odd degree. This two-dimensional algebra also admits a nondegenerate pairing of degree $-|x|$  so that $\langle x,1\rangle=\langle 1,x\rangle=1$ and pairings of $1$ and $x$ with themselves are zero. This pairing, is invariant and so, is determined by a trace or a Berezin integral on $\Lambda(x)$.

All this generalizes straightforwardly to several dimensions; we denote by $\bk[[{\bf y}]][{\bf y}^{-1}]$ the following algebra:
\[
\bk[[{\bf y}]][{\bf y}^{-1}]:=
\bk[[y^1,\ldots,y^m]][({y^1})^{-1},\ldots,({y^m})^{-1}]
\]
where $y^i,i=1,\ldots, m$ are variables of even degree. $\bk[[{\bf y}]][{\bf y}^{-1}]$ is clearly a certain completion of the tensor product of $n$ copies of $\bk[[y^i]][{(y^i)}^{-1}], i=1,\ldots, m$. It inherits an invariant inner product: for $f,g\in \bk[[{\bf y}]][{\bf y}^{-1}]$ their inner product $\langle f,g\rangle$ is the coefficient at the monomial $({y^1})^{-1}\cdots (y^m)^{-1}$ of the product $fg$. {The degree of the inner product $\langle -,-\rangle$ is ${|\bf y|}:=|y^1|+\ldots+|y^m|$}. This inner product is similarly determined by a nondegenerate trace, or a formal integral on $\bk[[{\bf y}]][{\bf y}^{-1}]$.
 Similarly, \[{\bf \Lambda}({\bf x}):=\Lambda(x^1,\ldots, x^n)\]
 is the exterior algebra on the variables $x^i,i=1,\ldots, n$ of odd degree. As a product of graded Frobenius algebras, $\bf \Lambda(x)$, it is itself a graded Frobenius algebra whose inner product is determined by a trace or a Berezin integral on $\Lambda(x^i), i=1,\ldots, n$ and {the degree of the nondegenerate inner product is ${-|\bf x|}:=-(|x^1|+\ldots+|x^n|)$.}

Finally, it is clear that the tensor product $\bk[[{\bf y}]][{\bf y}^{-1}]\otimes {\bf \Lambda(x)}\cong\bk[[{\bf y}]][{\bf y^{-1}}][{\bf x}]$ also possesses a nondegenerate inner product or a trace (a formal integral). For $f\in \bk[[{\bf y}]][{\bf y}^{-1}]\otimes {\bf \Lambda(x)}$ we will write this trace as $f\mapsto\int f\in \bk$.

\begin{rem}
	Consider the subalgebra $\bk[[{\bf y}]]$ inside $\bk[[{\bf y}]][{\bf y}^{-1}]\otimes {\bf \Lambda(x)}$. It is clear that the given inner product pairs it nondegenerately with ${\bf y}^{-1}\bk[{\bf y^{-1}}]$, the subspace in $\bk[[{\bf y}]][{\bf y}^{-1}]$ consisting of linear combinations of monomials in $(y^i)^{-1}$ having at least one factor of each $(y^i)^{-1}$. Moreover, under this pairing, $\bk[[{\bf y}]]$ is identified with the linear dual of ${\bf y}^{-1}\bk[{\bf y^{-1}}]$  and the latter subspace is identified with the \emph{continuous dual} of $\bk[[{\bf y}]]$. Also, the restriction of $\langle,\rangle$ on the finite-dimensional space $\bf\Lambda(\bf x)$ is nondegenerate; more specifically, $\bf \Lambda(\bf x)$ is identified with the linear dual of ${\bf \Lambda}({\bf x})[|\bf x|]$.  Putting this together, we obtain an isomorphism
	\begin{equation}\label{Iso-coal-pc-alge}
	\big({\bf y}^{-1}\bk[{\bf y^{-1}}]\otimes{\bf \Lambda({\bf x}})\big)^*\cong \bk[[{\bf y}]][-|{\bf y}|]\otimes{\bf \Lambda({\bf x}})[|\bf x|]
	\end{equation}
	induced by the inner product $\langle,\rangle$. This implies, in particular, that  $\big({\bf y}^{-1}\bk[{{\bf y}^{-1}}]\otimes \Lambda({\bf x})\big)[|\bf x|-|\bf y|]$ has a graded coalgebra structure, obtained as the continuous dual of the pc algebra structure of $\bk[[{\bf y}]]\otimes\bf\bf{\Lambda(\bf x)}$.
\end{rem}
\emptycomment{
A $\bk$-linear derivation (or a vector field) $\xi$ of the algebra $\huaS_n\otimes \Lambda_k$ is determined by its values on the generators $x_i, y_l, i=1,\ldots, n, l=1,\ldots k$ and can therefore be written as a sum
\[
\xi=\sum_{i=1}^n f_i\frac{\partial}{\partial x_i}+\sum_{l=1}^kg_l\frac{\partial}{\partial y_l}
\]
where $f_i, g_l\in \bk[[{\bf x}]][{\bf x^{-1}}][{\bf y}]$. Here the degree of the vector field $\xi$ is $|f_i|-|x_i|=|g_l|-|y_l|$.

\begin{defi}
	The \emph{divergence} of a vector field $\xi$ is defined as
	\[\nabla(\xi)=\sum_{i=1}^n \frac{\partial f_i}{\partial x_i}+\sum_{l=1}^k(-1)^{|g_l|}\frac{\partial g_l}{\partial y_l}.\]
	One can observe that the degree of $\nabla(\xi)$ is $|\xi|$.
	It is easy to see that   $\nabla(\xi)$ is defined invariantly and that it obeys the following well-known rules for any vector field $\xi$ and $f\in \bk[[{\bf x}]][{\bf x^{-1}}][{\bf y}]$:
	
	\begin{equation}\label{eq:divergenceproperty}
		\nabla(f\xi)=f\nabla\xi+(-1)^{|f||\xi|}\xi(f).
	\end{equation}
\end{defi}
}
The following formal analogue of Stokes' theorem is well-known (and easy to prove).
\begin{prop}\label{Stokes-theorem}
	For any vector field $\xi$, one has $\int\nabla(\xi)=0$. Conversely, if $\int f=0$, then $f=\nabla(\xi)$ for some vector field $\xi$.
\end{prop}

Given a vector field (a derivation) $\xi$ viewed as an operator on $\bk[[{\bf y}]][{\bf y^{-1}}][{\bf x}]$, one can ask for a formula for the \emph{adjoint operator} $\xi^*$ defined by the formula
\[
\langle \xi^*(f),g\rangle=-(-1)^{|\xi||f|}\langle f,\xi(g)\rangle
\]
where $f,g\in \bk[[{\bf y}]][{\bf y^{-1}}][{\bf x}]$.
\begin{prop}\label{prop:adjoint}
	Given a vector field $\xi$, its adjoint operator is given by the formula $\xi^*(f)=\xi(f)+\nabla(\xi)f$, where $ f\in \bk[[{\bf y}]][{\bf y^{-1}}][{\bf x}]$.
\end{prop}
\begin{proof}
	Expand the expression $\nabla(fg\xi)$ taking into account Stokes' formula and (\ref{eq:divergenceproperty}). For any vector field $\xi$ and functions $f,g\in \bk[[{\bf y}]][{\bf y^{-1}}][{\bf x}]$, we have
	\begin{eqnarray*}
		\int\nabla(fg\xi)&\stackrel{\eqref{eq:divergenceproperty}}{=}&\int (fg\nabla\xi+(-1)^{(|f|+|g|)|\xi|}\xi(fg))\\
		&=&\int fg\nabla\xi+(-1)^{(|f|+|g|)|\xi|}\int\xi(f)g+(-1)^{|g||\xi|}\int f\xi(g)\\
		&=&(-1)^{|\xi||g|}\langle f\nabla\xi,g\rangle+(-1)^{(|f|+|g|)|\xi|}\langle\xi(f),g\rangle+(-1)^{|g||\xi|}\langle f,\xi(g)\rangle.
	\end{eqnarray*}
	By Proposition \ref{Stokes-theorem}, one can deduce that
	\begin{eqnarray*}
		\langle f,\xi(g)\rangle{=}-\langle f\nabla\xi,g\rangle-(-1)^{|\xi||f|}\langle\xi(f),g\rangle.
	\end{eqnarray*}
	Finally, by the definition of the adjoint operator of $\xi^*$, we obtain that
	\begin{eqnarray*}
		\langle \xi^*(f),g\rangle=(-1)^{|\xi||f|}\langle f\nabla\xi,g\rangle+\langle\xi(f),g\rangle
	\end{eqnarray*}
	which implies that $\xi^*(f)=\xi(f)+\nabla(\xi)f$ as claimed.
\end{proof}
Let us write $\bf n$ for the multi-index $n_1,\ldots, n_m$ and denote by $(\bf y^n)$ the ideal in $\bk[[{\bf y}]]$ generated by
the elements $y^{n_1},\ldots, y^{n_m}$.
We have a finite-dimensional $\bk[[{\bf y}]]\otimes\bf\Lambda(\bf x)$-module $\bk[[{\bf y}]]/({\bf y}^{\bf{n}})\otimes \bf\Lambda(\bf x)$. Then the following result holds.
\begin{cor}\label{cor:adjoint}
Let $\xi$ be a derivation of $\bk[[{\bf y}]]\otimes \bf\Lambda(\bf x)$ that preserves the ideal $(\bf y^n)$ (and so, restricts to $\bk[[{\bf y}]]/({\bf y^n})\otimes \bf\Lambda(\bf x))$. Then there is an isomorphism of graded vector spaces
\[
\big(\bk[[{\bf y}]]/({\bf y^n})\otimes {\bf\Lambda}({\bf x})\big)^*\cong \big(\bk[[{\bf y}]]/({\bf y^n})\otimes {\bf\Lambda}({\bf x})\big)[|{\bf x}|-|{\bf y}|].
\]
Under this isomorphism, the adjoint to the operator $\xi$ acting on $\big(\bk[[{\bf y}]]/({\bf y^n})\otimes {\bf\Lambda}(\bf x)\big)^*$, corresponds to $\xi+\nabla(\xi)$ acting on $\big(\bk[[{\bf y}]]/({\bf y^n})\otimes {\bf\Lambda}(\bf x)\big)[|{\bf x}|-|{\bf y}|]$.
\end{cor}
\begin{proof}
	We have an isomorphism of $\bk[[{\bf y}]]\otimes \bf\Lambda(\bf x)$-modules:
	\begin{eqnarray*}
\big(\bk[[{\bf y}]]\otimes {\bf \Lambda}(\bf x)\big)^*&\stackrel{\eqref{Iso-coal-pc-alge}}{\cong}& \big({\bf y}^{-1}\bk[{\bf y}^{-1}]\otimes {\bf\Lambda}(\bf x)\big)[|{\bf x}|-|{\bf y}|].
		\end{eqnarray*}
	It follows that $\big(\bk[[{\bf y}]]/({\bf y^n})\otimes {\bf\Lambda}(\bf x)\big)^*$ is isomorphic to the submodule in $\big({\bf y}^{-1}\bk[{\bf y}^{-1}]\otimes {\bf\Lambda}(\bf x)\big)[|{\bf x}|-|{\bf y}|]$ annihilated by $(\bf y^n)$. This submodule is identified with $\big(\bk[[{\bf y}]]/({\bf y^n})\otimes {\bf\Lambda}(\bf x)\big)[|{\bf x}|-|{\bf y}|]$ and the claim about the action of the adjoint of $\xi$ follows from Proposition \ref{prop:adjoint}.
\end{proof}

\section{Derived linear duality for comodules over a coalgebra}\label{sec:strange assumption}
In this section we develop some methods for the computation of a derived linear dual $d(M)$ of a dg comodule $M$ over a dg coalgebra $C$.

One way to do it (used in Section \ref{sec:technical}) is to pass from $M$ to the Koszul dual module $M^!$ over the cobar-construction $\Omega(C)$ and use the fact that $d(M)$ corresponds to $\RHom_{\Omega (C)}(M^!,\Omega(C))$; if $C$ is the dg coalgebra $\B(A)$ or $\CE_\bullet(\g)$ for a dg algebra $A$ or a dg Lie algebra $\g$ respectively, then one can use $A$ or $\U(\g)$ in place of $\Omega(C)$. In this section we discuss ways of computing $d(M)$ directly, staying in the model category of dg $C$-comodules.

It will be convenient to switch from $C$-comodules to pc $A:=C^*$-modules.
Note that the category of pc $A$-modules is anti-equivalent to the category of $C$-comodules and so, it has a model structure whose homotopy category is anti-equivalent to that of $D^{\co}(C)$. We will denote this homotopy category by $D^{\co}(A)$ and refer to it as the coderived category of $A$. Let $M$ be a dg $C$-comodule. Then one can assume that $M$ is the filtered colimit of its finite-dimensional dg sub-comodules: $M=\colim_{\alpha}M_\alpha$ with $\dim M_\alpha<+\infty$.
Homotopy limits of a $C$-comodules correspond to homotopy colimits of the dual pc $A$-modules and so we are interested in computing $(\hocolim_{\text{pc}})_{\alpha}N_\alpha$. Here $N_\alpha:=M_\alpha^*$ is the pc $A$-module which is  linearly dual (in the ordinary, naive sense) to the finite-dimensional comodule $M_\alpha$. Then the  derived linear dual of $M$ is $d(M)=\big((\hocolim_{\text{pc}})_{\alpha}N_\alpha\big)^{*}$.

To this end, let us introduce the notion of a derived profinite completion. The category of pc $A$-modules has a forgetful functor $U$ into the category of discrete $A$-modules. It takes weak equivalences of pc $A$-modules into quasi-isomorphisms and preserves fibrations (which are in both cases surjective maps). It also has a left adjoint, the profinite completion functor
$M\mapsto M^\wedge:=\lim_\alpha M_\alpha$ where the inverse limit is taken over all finite-dimensional $A$-module quotients $M_\alpha$ of $M$. It follows that the forgetful functor $U$ is right Quillen and the profinite completion is left Quillen. We will denote its left derived functor as $M\mapsto {\huaL^\wedge}(M)$; it can be computed by replacing $M$ by a quasi-isomorphic cofibrant $A$-module $LM$, and then applying the profinite completion functor. Note that although the counit of the adjunction
$M\to M^{\wedge}$
is an isomorphism of pc modules (completing a pc module amounts to doing nothing), it is not true that the \emph{derived} counit map $\huaL^\wedge (M)\to M$  is a weak equivalence because a cofibrant pc $A$-module need not be cofibrant as a discrete $A$-module.
\begin{example}\label{ex:acyclicmodule}
	Let $A:=\Lambda(x)$, the exterior algebra on one generator $x$ with $|x|=1$ and zero differential. Let $M$ be the free $A$-module of rank $1$ but with the nontrivial differential $d(1)=x$. As a pc $A$-module, $M$ is cofibrant and it is nontrivial, e.g. $M\otimes^{L,\text{II}}_A\bk \cong \bk\neq 0$ where $\otimes^{L, \text{II}}_A$ stands for the derived tensor product on pc $A$-modules. At the same time, as a discrete dg $A$-module, $M$ is not cofibrant, and its cofibrant replacement is zero since $M$ is acyclic. Thus, ${\huaL^\wedge} (M)\simeq 0$.
\end{example}

However, if (the weak equivalence class of) the pc $A$-module $M$ lies in the thick subcategory of the coderived category of  pc $A$-modules generated by $A$ itself, this difficulty does not arise. Moreover, let $\huaTT$ be a triangulated category and $X$ be a set
of objects of $\huaTT$.  We denote $\huathick_{\huaTT}(X)$ the thick subcategory of $\huaTT$ generated by $X$, which is the smallest full triangulated subcategory
of  $\huaTT$ containing $X$ and closed under direct summands. It can be obtained as the
closure of $X$ under shifts, cones and direct summands.
\begin{lem}\label{lem:thick}
	Let $M$ be a pc module over a pc algebra $A$ that lies in
	$\huathick_{\text{D}^{\co}(A)}(A)$.
	Then $\huaL^\wedge (M)\simeq M$.
\end{lem}
\begin{proof} Note that, since the forgetful functor from pc $A$-modules to discrete $A$-modules preserves finite direct sums, homotopy fibre sequences and retracts, $M$ lies in $\huathick_{{D}(A)}(A)$ viewed as a \emph{discrete} $A$-module.
	The statement of the lemma is obvious if $M=A$ since  $A$ is cofibrant as a discrete module over itself. The general case then follows since the functor $\huaL^\wedge$ preserves finite direct sums, homotopy fibre sequences and retracts.
\end{proof}

Let us now consider the case when a pc dg algebra $A$ is finite-dimensional. Admittedly, this is not the case especially relevant to this paper since universal enveloping algebras are typically infinite-dimensional; however this is when derived profinite completion admits a very explicit description and so we record here the relevant conclusions.

The following proposition describes the derived completion functor on $A$-modules. Recall that we denoted by $U$ the forgetful functor from pc $A$-modules to $A$-modules.

\begin{prop}\label{prop:doubledual}
	For any dg module $M$ over a finite-dimensional dg algebra $A$ we have a natural isomorphism
	\begin{equation}\label{eq:doubledual}
		M^\wedge\cong \big(U(M^*)\big)^*.
	\end{equation}
\end{prop}
\begin{proof}
	Note that the functor of linear duality $M\mapsto M^*$ determines an equivalence of categories $$A\MMod\to ({\text {pc}}~A\MMod)^{\op}$$
	with the quasi-inverse functor being the continuous dual $N\mapsto N^*$ where $N$ is a dg pc $A$-module. (Here finite-dimensionality of $A$ is used: if $\dim A=\infty$ then $M^*$ is not necessarily a pc $A$-module.)
	
	Thus, to demonstrate (\ref{eq:doubledual}), it suffices to show that $(M^\wedge)^*\cong M^*$.
	We have
	
	\[
	(M^\wedge)^*\cong\varinjlim (M_\alpha)^*
	\]
	where $M_\alpha$ range over all finite-dimensional dg $A$-module quotients of $M$ and thus, $M_\alpha^*$ range over all finite-dimensional dg $A$-submodules of $M^*$. Since $A$ is finite-dimensional, any dg $A$-module is a nested union of its finite-dimensional submodules and we are done.
\end{proof}
Given a dg $A$-module $M$, denote by $\widetilde{M}$ the cofibrant replacement of $M$. Then the following result is immediate from Proposition \ref{prop:doubledual}.
\begin{cor}
	For a dg module $A$ over a finite-dimensional dg algebra $A$, its derived profinite completion ${\bf L}^\wedge(A)$ is weakly equivalent to the dg pc $A$-module $(U(\widetilde{M})^*)^*$.
\end{cor}

The coderived category of pc dg  $A$-modules is related to the opposite to the derived category of $A$ (not necessarily assumed finite-dimensional):
\begin{prop}\label{prop:deriveddoubledual}
	The functor $M\mapsto M^*:A\MMod\to ({\text {pc}}~A\MMod)^{\op}$ is a left Quillen functor whose adjoint right Quillen functor is the continuous dual $N\mapsto N^*$. The latter preserves arbitrary weak equivalences (and thus, coincides with its own left derived functor). The derived unit map $M\to \left((\widetilde{M})^*\right)^*$ is a quasi-isomorphism.
\end{prop}
\begin{proof}
	The functor $N\mapsto N^*:({\text {pc}}~A\MMod)^{\op}\to A\MMod$ preserves quasi-isomorphisms and thus, \emph{a fortiori}, takes weak equivalences of pc dg $A$-modules to quasi-isomorphisms of $A$-modules. It also takes fibrations in $({\text {pc}}~A\MMod)^{\op}$ (which are injective maps) to surjections (i.e. fibrations in $\MMod-A$). Thus, it is a right Quillen functor and its adjoint is left Quillen. The statement about the derived unit map is likewise clear.
\end{proof}

\begin{rem}
	The Quillen adjunction of Proposition \ref{prop:deriveddoubledual} is \emph{not} a Quillen equivalence despite being an equivalence of 1-categories (in particular, it does not follow that the homotopy categories of $A\MMod$ and $({\text {pc}}~A\MMod)^{\op}$ are equivalent). Indeed, the module $M$ of Example \ref{ex:acyclicmodule} is not weakly equivalent to zero as a pc $\Lambda$-module, but is acyclic and so is its (continuous) dual $M^*$, which makes him trivial in the derived category of $\Lambda$.
\end{rem}	

Next, note that any dg comodule $M$ over a dg coalgebra $C$ naturally has the structure of a $C^*$-module (though not necessarily pc). This is clear when $M$ is finite-dimensional; in this case the coaction map $M\to C\otimes M$ is equivalent to a continuous dg algebra map $C^*\to\End(M^*)\cong \End(M)$. The general case follows since $M$ is a union of its finite-dimensional subcomodules.
\begin{example}
	Consider the following short exact sequence of $\bk[[x]]$-modules
	\[
	\bk[[x]]\to  \bk[[x]][x^{-1}]\to x^{-1}\bk[x^{-1}].
	\]
	Here the $\bk[[x]]$-module structure on $x^{-1}\bk[x^{-1}]$ is obtained by the isomorphism $x^{-1}\bk[x^{-1}]\cong \bk[[x]][x^{-1}]/\bk[[x]]$. Furthermore, $x^{-1}\bk[x^{-1}]$ is a coalgebra whose linear dual is the pc algebra $\bk[[x]]$ (recall that there is a pairing between $x^{-1}\bk[x^{-1}]$ and $\bk[[x]]$ given by the residue of the product of two power series). As such, it is a comodule over itself and also a non-pc $\bk[[x]]$-module.
\end{example}

The following result gives a description of the derived linear dual of a comodule in favorable cases.
\begin{prop}\label{prop:derivedualcriterion}
	Let $C$ be a dg coalgebra and $M$ be a $C$-comodule. Assume that $M=\bigcup_\alpha M_\alpha$, the nested union of finite-dimensional $C$-subcomodules $M_\alpha$ which lie in the thick subcategory of the coderived category of $C$-comodules generated by $C$.  Then
	$$d(M)\underset{\text{w.e.}}{\simeq} \big(\huaL^\wedge (M)\big)^{*}.$$
	In other words, we  view $M$ as a $C^*$-module, take its derived profinite completion (making it into a pc $C^*$-module) and then its continuous dual, converting back to a $C$-comodule which is weakly equivalent to the derived linear dual of $M$.	
\end{prop}
\begin{proof}
	Let $M^*$ be the pc $C^*$-module corresponding to the $C$-comodule $M$; then $M^*\cong \lim M^*_\alpha\simeq \holim M^*_\alpha$.  Since $M_\alpha\in \huathick_{{D}^{co}(C)}(C)$ we have that, dually, $M^*_\alpha\in \huathick_{{D}^{co}(C^*)}(C^*)$ and therefore by Lemma \ref{lem:thick}, ${\huaL^\wedge}(M^*_\alpha)\simeq M^*_\alpha$. Since the derived completion functor commutes with the homotopy colimits (being left Quillen), we obtain:
	\begin{align*}
		{\hocolim}_{\text{pc}} (M^*_\alpha)^*&\cong {\hocolim}_{\text{pc}} (M_\alpha)\\
		&\underset{\text{\em w.e.}}{\simeq} \huaL^\wedge (\hocolim M_\alpha)\\
		&\underset{\text{\em w.e.}}{\simeq} \huaL^\wedge (\colim M_\alpha)\\
		&\underset{\text{\em w.e.}}{\simeq} \huaL^\wedge (M).
	\end{align*}
	Here $\hocolim_{\text {pc}}$ refers to the homolotopy colimit taken in the category of pc $C^*$-modules whereas unmarked $\hocolim$ is the homotopy colimit in the discrete category (and so, it is weakly equivalent to the ordinary colimit since it is taken along a system of inclusions). Thus, one obtain that $d(M)=\big((\hocolim_{\text{pc}})_{\alpha}M_\alpha\big)^{*}\underset{\text{\em w.e.}}{\simeq}  \big(\huaL^\wedge (M)\big)^{*}$ as claimed.
\end{proof}
\begin{example}\label{ex:polynomial}
	Let $C^*=\bk[[x]]$ (so that $C\cong x^{-1}\bk[x^{-1}]$) and $M=C$. We have $C^*\cong \underset{\text{\em n}}{\lim}~\bk[x]/x^n$ and the $\bk[[x]]$-module $\bk[x]/x^n$ admits a length two free resolution as a pc $\bk[[x]]$-module, so the conditions of Proposition \ref{prop:derivedualcriterion} are satisfied. We conclude that $$d(x^{-1}\bk[x^{-1}])\simeq \big(\huaL^\wedge (x^{-1}\bk[x^{-1}])\big)^{*}.$$ This, however, does not quite finish the calculation since the cofibrant replacement of the discrete $\bk[[x]]$-module $x^{-1}\bk[x^{-1}]$ is not very nice. Instead, we will use the fact that $x^{-1}\bk[x^{-1}]$ is quasi-isomorphic to the homotopy cofiber of the inclusion $\bk[[x]]\to \bk[[x]][x^{-1}]$ (in fact, it is isomorphic to the actual cofiber). Let $M$ be a finite-dimensional $\bk[[x]]$-module; then $\RHom_{\bk[[x]]}(x^{-1}\bk[x^{-1}],M)$ is quasi-isomorphic to the homotopy limit of the system $M\to M\to\ldots$, where the maps are given by the multiplication by $x$. Then $\varprojlim \HH^\bullet(M)=0$ and because $\dim <\infty$, it follows that no $\lim^1$-terms are present and $\RHom_{\bk[[x]]}(x^{-1}\bk[x^{-1}],M)\simeq 0$ (we thank an anonymous referee for suggesting this argument).
	
	It follows that ${\huaL^\wedge}(\bk[[x]][x^{-1}])=0$. Now note that ${\huaL^\wedge}(\bk[[x]])\simeq \bk[[x]]$  and, since the derived completion functor (being left Quillen) preserves homotopy cofiber sequences, we conclude that ${\huaL^\wedge}(x^{-1}\bk[x^{-1}])\simeq \bk[[x]][1]$. Finally, $d(C)\simeq (\bk[[x]][1])^{*}\cong x^{-1}\bk[x^{-1}][-1]$.
\end{example}
This (seemingly paradoxical) example, using a more ad hoc method, was computed in \cite{BCL}. Of course, it could also be derived from Theorem \ref{derived-CE-comod} since $C^*\cong\bk[[x]]\cong \CE^\bullet(\g)$ where $\g$ is the one-dimensional Lie algebra concentrated in degree 1.
However, it can be generalized in ways not covered by Theorem \ref{derived-CE-comod} as we will demonstrate.

Let $\g$ be a finite-dimensional $L_\infty$-algebra with representing pc dg algebra $\CE^\bullet(\g)$. We will make the following (admittedly very strong) assumption on $\g$.

\begin{hyp}[H]\hypertarget{hyp:H} The image of all higher $L_\infty$-brackets $\Lambda^k\g\to \g$
	is contained in the even graded component of $\g$.
\end{hyp}

We are interested in computing the derived dual to $\CE_\bullet(\g)$ as a comodule over itself. The following result holds.

\begin{prop}\label{prop:CEdual}
	Let $\g$ be a finite-dimensional $L_\infty$-algebra $(\g,\ell)$ satisfying Hypothesis \hyperlink{hyp:H}{(H)}. Then $d(\CE_\bullet(\g))$, the derived linear dual of $\CE_\bullet(\g)$, is weakly equivalent to $\big(\huaL^\wedge (\CE_\bullet(\g))\big)^{*}$  as a comodule over itself.
\end{prop}
\begin{proof}
	We simply check the conditions of Proposition \ref{prop:derivedualcriterion}. Let the even degree part of $\g[1]$ have dimension $m$ and its odd part have dimension $n$. Then, disregarding the differential, $\CE^\bullet(\g)$ is isomorphic to $\bk[[y^1,\ldots, y^m]]\otimes \Lambda(x^1,\ldots, x^n)$. Note that the Hypothesis \hyperlink{hyp:H}{(H)} ensures that all even generators $y^i,i=1,\ldots, m$ are cocycles with respect to the derivation $\ell$. Let  $I_j$ be the ideal in $\CE^\bullet(\g)$ generated by $(y^1)^j,\ldots, (y^m)^j$ and note that (since the generators $y^i$ are cocycles) this is a dg ideal. Note that
	\[
\CE^\bullet(\g)\cong\lim_j \CE^\bullet(\g)/I_j
	\]
	and $\CE^\bullet(\g)/I_j$ are finite-dimensional $\CE^\bullet(\g)$-modules. The dg $\CE^\bullet(\g)$-module $$\bk[[y^1,\ldots, y^m]]/(y^1)^j\otimes \Lambda(x^1,\ldots, x^n)$$ is a cofiber of the multiplication by $(y^1)^j$ as a self-map of dg $\CE^\bullet(\g)$ and therefore, it belongs to the thick subcategory of $D^{\co}(\CE^\bullet(\g))$ generated by $\CE^\bullet(\g)$. By an obvious induction, the dg $\CE^\bullet(\g)$-module $\CE^\bullet(\g)/I_j$ likewise belongs to the same thick subcategory of $D^{\co}(\CE^\bullet(\g))$.
	
	So the conditions of Proposition \ref{prop:derivedualcriterion} are fulfilled and the desired statement follows.
\end{proof}

Let $A$ be a pc local commutative dg algebra and $z$ be a cocycle in $A$. Inverting $z$ results in a dg algebra $A[z^{-1}]=A\otimes_{\bk[z]}\bk[z,z^{-1}]$ (not pc). Furthermore, given a pc  dg $A$-module $M$, define its localization at $z$ as $M\otimes_{A}A[z^{-1}]$; this will be a non-pc dg $A$-module. These constructions are all homotopy invariant in a suitable sense: $A$ is dual to a conilpotent dg coalgebra $A^*$  and the formation of $A[z^{-1}]$ only depends, up to quasi-isomorphism, on the quasi-isomorphism type of $A$ and,\emph{ a fortiori}, on the weak equivalence type of the coalgebra $A^*$. Similarly, the pc module $M$ is dual to a comodule $M^*$ over the coalgebra $A^*$. Owing to flatness of $A[z^{-1}]$ over $A$, the $A$-module $M[z^{-1}]$ only depends, up to quasi-isomorphism, on the quasi-isomorphism type of $M$ and thus, \emph{a fortiori}, on the weak equivalence type of the pc $A^*$-comodule $M^*$.

\begin{lem}\label{lem:completelocal1}
Given a pc dg algebra $A$ and a pc dg $A$-module $M$,	the derived profinite completion $\huaL^\wedge (M[z^{-1}])$ of the $A$-module $M[z^{-1}]$ is weakly equivalent to zero.
\end{lem}	
\begin{proof}
	Note that the profinite completion of  $M[z^{-1}]$ is formed by taking the inverse limit of all finite-dimensional $A$-module  quotients of $M[z^{-1}]$. But any $A$-module into which $M[z^{-1}]$ maps will have a $\bk$-subspace isomorphic to $\bk[z,z^{-1}]$ and thus, cannot be finite-dimensional unless it is zero. Thus, the profinite completion of $M[z^{-1}]$ is zero. Arguing as in Example \ref{ex:polynomial}, we see that for any finite-dimensional $A$-module $N$, the dg vector space $\RHom_A(M[z^{-1}],N)$ is acyclic, in other words there are no nontrivial morphisms in the derived category of $A$ from $M[z^{-1}]$ into a finite-dimensional $A$-module. By adjunction, there are no nontrivial maps in the coderived category of $A$ (i.e. the homotopy category of pc $A$-modules) from $\huaL^\wedge (M[z^{-1}])$ into any finite-dimensional $A$-module. That means that  $\huaL^\wedge (M[z^{-1}])$ is weakly equivalent to zero as a pc $A$-module.
\end{proof}
The following proposition establishes Conjecture \ref{conjecture-comod}  in a variety of situations.
\begin{prop}\label{prop:CEdual'}
	Let $(\g,\ell)$ be a finite-dimensional $L_\infty$-algebra satisfying Hypothesis \hyperlink{hyp:H}{(H)}. Then, Conjecture \ref{conjecture-comod} holds for $\g$.
\end{prop}
\begin{proof}
	Consider the graded commutative pc algebra $\CE^\bullet(\g)$, linear dual to the coalgebra $\CE_\bullet(\g)$ supplied.

The given assumption on the $L_\infty$-structure of $\g$ implies that the even elements of $(\g[1])^*\in\CE^\bullet(\g)$ are cocycles under the differential $\ell$.
Choose a basis ${\bf x,y}$ of $(\g[1])^*$ where, as before, ${\bf y}=\{y^1,\ldots,y^m\}$ are even basis elements of $(\g[1])^*$, and ${\bf x}=\{x^1,\ldots,x^n\}$ are the odd ones; in other words, we have $\CE^\bullet(\g)\cong \bk[[{\bf y}]]\otimes\Lambda(\bf x)$. The vector field $\ell$ has the property that $\ell({\bf y})=0$.
We have the following commutative diagram of pc $\CE^\bullet(\g)$-modules:
		\begin{equation}\label{eq:diagram}	\xymatrix{
		{\CE^\bullet}(\g)^{[\nabla(\ell)]}\ar^{\cong}[r]\ar^{y^1}[d]&{\CE^\bullet}(\g)^{[\nabla(\ell)]}\ar^{\cong}[r]\ar^{(y^1)^2}[d]&\ldots&{\CE^\bullet}(\g)^{[\nabla(\ell)]}\ar^{\cong}[r]\ar^{(y^1)^k}[d]&\ldots \\
		{\CE^\bullet}(\g)^{[\nabla(\ell)]}\ar^{y^1}[r]\ar[d]&{\CE^\bullet}(\g)^{[\nabla(\ell)]}\ar[d]\ar^{y^1}[r]&\ldots &{\CE^\bullet}(\g)^{^{[\nabla(\ell)]}}\ar^{y^1}[r]\ar[d]&\ldots\\
		{\CE^\bullet}(\g)^{[\nabla(\ell)]}/y^1\ar^{y^1}[r]&{\CE^\bullet}(\g)^{[\nabla(\ell)]}/(y^1)^2\ar^-{y^1}[r]&\ldots&{\CE^\bullet}(\g)^{[\nabla(\ell)]}/(y^1)^k\ar^-{y^1}[r]&\ldots
	}
	\end{equation}
	Note that by Hypothesis \hyperlink{hyp:H}{(H)}, $y^1$ is a $\ell$-cocycle in $\CE^\bullet(\g)$ and so, the action by $y^1$ in any dg $\CE^\bullet(\g)$-module is a dg $\CE^\bullet(\g)$-module map.

The columns of diagram \ref{eq:diagram} are homotopy cofibers of pc $\CE^\bullet(\g)$-modules and, therefore, so are their homotopy colimits (taken in the model category of pc dg $\CE^\bullet(\g)$-modules). Denote the homotopy colimit of the middle row by $X$; we claim that $X$ is weakly equivalent to zero as a pc $\CE^{\bullet}(\g)$-module. Indeed, applying to it the functor $\RHom_{\text{pc}\CE^\bullet(\g)}(-,\bk)$, we obtain the sequence $\bk\to\bk\to\ldots\to\bk\ldots\to$ with zero connecting maps. That implies that $\RHom_{\text{pc}\CE^\bullet(\g)}(X,\bk)\simeq 0$ and since $\bk$ is a cogenerator of $D^{\co}\big(\CE^\bullet(\g)\big)$ that implies that $X\simeq 0$.

It follows that there is a weak equivalence of pc dg $\CE^\bullet(\g)$-modules
\[
{\hocolim}_{k}{\CE^\bullet}(\g)^{[\nabla(\ell)]}/(y^1)^k\simeq {\CE^\bullet}(\g)^{[\nabla(\ell)]}[1].
\]
Arguing similarly and applying induction, we conclude that 	
\[
{\hocolim}_{k_1,\ldots k_m}{\CE^\bullet}(\g)^{[\nabla(\ell)]}/[(y^1)^{k_1},\ldots, (y^m)^{k_m}]\simeq {\CE^\bullet}(\g)^{[\nabla(\ell)]}[m].
\]
We will rewrite this as
\[
{\hocolim}_{\bf k}{\CE^\bullet}(\g)^{[\nabla(\ell)]}/[({\bf y})^{\bf k}]\simeq {\CE^\bullet}(\g)^{[\nabla(\ell)]}[m].
\]
where ${\bf k}:=(k_1,\ldots, k_m)$.
As in the proof of Propostion \ref{prop:CEdual}, we see that
\[
\CE^\bullet(\g)\cong \lim _{\bf k}{\CE^\bullet}(\g)/[{(\bf y)}^{\bf k}]
\]
and by Corollary \ref{cor:adjoint} we have that
\[
{(\CE^\bullet}(\g)/[({\bf y})^{\bf k}])^*\cong {\CE^\bullet}(\g)^{[\nabla(\ell)]}/[({\bf y})^{\bf k}][\bf{|x|-|y|}]
\]
and so
\begin{align*}
	d(\CE^\bullet(\g))&\simeq \hocolim_{\bf k}({\CE^\bullet}(\g)/[{\bf y}^{\bf k}])^*\\
	&\simeq  {\CE^\bullet}(\g)^{[\nabla(\ell)]}[{m+\bf{|x|-|y|}}]\\
	&\simeq  {\CE^\bullet}(\g)^{[\nabla(\ell)]}[|\g|].
\end{align*}
Since the image of all higher $L_\infty$-brackets $\Lambda^k\g\to \g$
	is contained in the even graded component of $\g$, we deduce that $|\g|=|\HH^\bullet(\g)|$. Thus we obtain that 
\begin{eqnarray*}
d(\CE_\bullet(\g))\simeq d(\CE^\bullet(\g))^*\simeq \big({\CE^\bullet}(\g)^{[\nabla(\ell)]}[|\HH^\bullet(\g)|]\big)^*\simeq {\CE_\bullet}(\g)^{[\nabla(\ell)]}[-|\HH^\bullet(\g)|].
\end{eqnarray*}
This finishes the proof.
\end{proof}

\section{Examples}
\begin{example}
Consider the two dimensional $L_\infty$-algebra $(\g,\ell)$ whose representative pc dg algebra is of the form $\CE^\bullet(\g)=\bk[[y]]\otimes\Lambda(x)$ with $|y|=2$ and $|x|=2n+1$. The CE differential has the form $\ell=y^{n+1}\frac{\partial}{\partial x}$, i.e. $\ell(y)=0$ and $\ell(x)=y^{n+1}$. The conditions of Proposition \ref{prop:CEdual'} are satisfied; $\nabla(\ell)=0$ and $|\g|=1+|x|-|y|=2n$. The dg coalgebra $\CE_\bullet(\g)$ is $(2n)$-Frobenius and $d(\CE_\bullet(\g))\simeq\CE_\bullet(\g)[-2n]$.

In fact, this can be seen from more elementary considerations. It is clear that the map $\CE^\bullet(\g)\to \bk[y]/y^{n+1}, x\mapsto 0, y\mapsto y$ is a filtered quasi-isomorphism and therefore, a weak equivalence of pc dg algebras. In other words, $\CE^\bullet(\g)$ has a finite-dimensional model and for the latter the derived dual is the same as the ordinary $\bk$-linear dual. Furthermore, for $\bk=\mathbb{Q}$, it is clear that  $\CE^\bullet(\g)$ is the Sullivan-de Rham minimal model of the space ${\mathbb C}P^n$. The coderived category of $\CE_\bullet(\g)$-comodules is equivalent to the derived category of $\infty$-local systems on ${\mathbb C}P^n$. It is CY, reflecting Poincar\'e duality on ${\mathbb C}P^n$.
\end{example}

\begin{example}
	Let $\g$ be a two-dimensional  $L_\infty$-algebra for which $\CE^\bullet(\g)=\bk[[x]]\otimes\Lambda(y)$. Here the degree of $x$ is zero and the degree of $y$ is $1$. The CE differential determining the $L_\infty$-structure on $\g$ is defined by
	\begin{eqnarray*}
		\ell=x^ny\frac{\partial}{\partial x},\,\,n\in \mathbb Z_{\ge 1}.
	\end{eqnarray*}
	One has $\nabla(\ell)=nx^{n-1}y$ and $|\g|=1-|x|+|y|=2$. The cochain complex of the $\CE^\bullet(\g)$ is given by
	\begin{eqnarray*}
		0\longrightarrow \bk[[x]]1 \stackrel{\ell}{\longrightarrow}\bk[[x]]y\longrightarrow 0.
	\end{eqnarray*}
	So the cohomology algebra of $\CE^\bullet(\g)$ is the graded algebra
	$$
	\ker(\ell)\oplus
	\frac{\bk[[x]]y}{\IM(\ell)}
	$$
	which is concentrated in degree $0$ and $1$. For the element $f=\sum_{i=0}^{+\infty}a_ix^i\in \bk[[x]]$, we have
	\begin{eqnarray}\label{eq:sqzero}
		\ell(f)=\sum_{i=0}^{+\infty}ia_ix^{n+i-1}y.
	\end{eqnarray}
	One deduces that $\ker(\ell)=\bk$ and $\IM(\ell)=x^n\bk[[x]]y$. Thus, is
\begin{eqnarray}\label{eq:sqzero-1}
\HH^\bullet (\CE^\bullet(\g))\cong \bk\oplus (\bk[x]/x^n)y.	
\end{eqnarray}
	More precisely, the cohomology algebra $\HH^\bullet (\CE^\bullet(\g))$ is the extension of $\bk$ by an $n$-dimensional square-zero ideal $(\bk[x]/x^n)y$ concentrated in degree 1.
	
	
	
	The cochain complex of the twisted module  $\CE^\bullet(\g)^{[\nabla(\ell)]}$ is given by
	\begin{eqnarray*}
		0\longrightarrow \bk[[x]] \stackrel{\ell^{[\nabla(\ell)]}}{\longrightarrow}\bk[[x]]y\longrightarrow 0.
	\end{eqnarray*}
	So the cohomology of $\CE^\bullet(\g)^{[\nabla(\ell)]}$ is the graded module
	$$
	\ker(\ell^{[\nabla(\ell)]})\oplus
	\frac{\bk[[x]]y}{\IM(\ell^{[\nabla(\ell)]})}.
	$$
	For element $f=\sum_{i=0}^{+\infty}a_ix^i\in \bk[[x]]$, we have
	\begin{eqnarray*}
		\ell^{[\nabla(\ell)]}(f)=\sum_{i=0}^{+\infty}(i+n)a_ix^{n+i-1}y.
	\end{eqnarray*}
	One deduces that $\ker(\ell^{[\nabla(\ell)]})=\{0\}$ and $\IM(\ell^{[\nabla(\ell)]})=x^{n-1}\bk[[x]]y$. Thus,  is
		\begin{equation}\label{eq:sqzero1}
\HH^\bullet(\CE^\bullet(\g)^{[\nabla(\ell)]})\cong (\bk[x]/{x^{n-1}})y,
\end{equation}
the $n$-dimensional vector space sitting in degree 1. Moreover $\HH^\bullet(\CE^\bullet(\g)^{[\nabla(\ell)]})$ is a graded $\HH^\bullet (\CE^\bullet(\g))$-module where the action of $\HH^\bullet (\CE^\bullet(\g))$, cf (\ref{eq:sqzero-1}) is through the augmentation into $\bk$.
	
	Note that Proposition \ref{prop:CEdual'} is not directly applicable to $\g$ since the CE differential $\ell$ is not zero when applied to
	the even generator $x$ (so Hypothesis \hyperlink{hyp:H}{(H)} is not satisfied). However, it turns out that the method applied in the proof of Proposition \ref{prop:CEdual'} still works  (because $\ell(x)$ is divisible by $x$ even though it is not zero).
	
		So consider the following diagram of $\CE^\bullet(\g)$-modules:
		\[	\xymatrix{
			{\CE^\bullet}(\g)^{[x^{n-1}y]}\ar^{\cong}[r]\ar^x[d]&{\CE^\bullet}(\g)^{[x^{n-1}y]}\ar^{\cong}[r]\ar^{x^2}[d]&\ldots&{\CE^\bullet}(\g)^{[x^{n-1}y]}\ar^{\cong}[r]\ar^{x^k}[d]&\ldots \\
			{\CE^\bullet}(\g)\ar^x[r]\ar[d]&{\CE^\bullet}(\g)^{[-x^{n-1}y]}\ar[d]\ar^-x[r]&\ldots &{\CE^\bullet}(\g)^{[-(k-1)x^{n-1}y]}\ar^-{x}[r]\ar[d]&\ldots\\
			{\CE^\bullet}(\g)/x\ar^x[r]&{\CE^\bullet}(\g)^{[-x^{n-1}y]}/x^2\ar^-x[r]&\ldots&{\CE^\bullet}(\g)^{[-(k-1)x^{n-1}y]}/x^k\ar^-{x}[r]&\ldots
		}
		\]
		The columns of the above diagram are cofibre sequences and therefore, so are their hocolims. It follows that $d(\CE^\bullet(\g))$ is weakly equivalent to the cofibre of the map $${\CE^\bullet}(\g)^{[x^{n-1}y]}[1]\to \hocolim_k {\CE^\bullet}(\g)^{[-k x^{n-1}y]}[1].$$
		
		We claim that $\hocolim_k {\CE^\bullet}(\g)^{[-k x^{n-1}y]}=0$. Note applying the functor $\Hom_{\CE^\bullet(\g)}(-,\bk)$ to the inverse system ${\CE^\bullet}(\g)^{[-k x^{n-1}y]}$ we obtain the direct system $\bk\to \bk\to\ldots$ and all the maps in it are zero. It follows that
		$\RHom_{\CE^\bullet(\g)}(\hocolim_k {\CE^\bullet}(\g)^{[-k x^{n-1}y]},\bk)=0$ and since $\bk$ is the cogenerator of the coderived category of pc $\CE^\bullet(\g)$-modules, the conclusion follows.
		
		We conclude that $d(\CE^\bullet(\g))\simeq \CE^\bullet(\g)^{[x^{n-1}y]}[2]$
		 and so $d(\CE_\bullet(\g))\simeq {\CE}_\bullet(\g)^{[x^{n-1}y]}[-2]$ confirming Conjecture \ref{conjecture-comod} in this case.
	\end{example}

	\begin{example}
		Let $\g$ be a 3-dimensional $L_\infty$-algebra whose representing pc dg algebra is $\bk[[x,y]]\otimes\Lambda(z)$. Here the degrees  of $x$ and $y$ are zero and the degree of $z$ is $1$. The CE differential $\ell$ is defined by
		\begin{eqnarray*}
			\ell=y^nz\frac{\partial}{\partial x},\,\,n\in \mathbb Z_{\ge 1}.
		\end{eqnarray*}
		One has $\nabla(\ell)=0$ and $|\g|=2-|x|-|y|+|z|=3$. The cochain complex of the $\CE^\bullet(\g)$ is given by
		\begin{eqnarray*}
			0\longrightarrow \bk[[x,y]] \stackrel{\ell}{\longrightarrow}\bk[[x,y]]z\longrightarrow 0.
		\end{eqnarray*}
		So the cohomology algebra of $\CE^\bullet(\g)$ is the graded algebra
		$$
		\ker(\ell)\oplus
		\frac{\bk[[x,y]]z}{\IM(\ell)}
		$$
		which is concentrated in degree $0$ and $1$. For element $f=\sum_{i,j=0}^{+\infty}a_{ij}x^iy^j\in \bk[[x,y]]$, we have
		\begin{eqnarray*}
			\ell(f)=\sum_{i,j=0}^{+\infty}ia_{ij}x^{i-1}y^{n+j}z.
		\end{eqnarray*}
		Moreover, one can deduce that $\ker(\ell)=\bk[[y]]$ and $\IM(\ell)=y^n\bk[[x,y]]z$. Thus the cohomology algebra of $\CE^\bullet(\g)$ is
		$$
		\bk[[y]]\oplus \frac{\bk[[x,y]]}{y^n\bk[[x,y]]}z,
		$$
and the multiplication is given by
		\begin{eqnarray*}
			y^i\cdot y^j=y^{i+j},i,j=0,\ldots,~y^i\cdot x^j[y^k]z=x^j[y^{i+k}]z,i,j,k=0,\ldots,~x^j[y^k]z\cdot x^{j'}[y^{k'}]z=0.
		\end{eqnarray*}
What about $d(\CE^\bullet(\g))$? For $n=1$, $\g$ is a graded Lie algebra (higher $L_\infty$-structure maps vanish) and Theorem \ref{thm:super-g-trace} applies giving $d(\CE_\bullet(\g))\simeq\CE_\bullet(\g)[-3]$. For a general $n$, Conjecture \ref{conjecture-comod}
		 would give the same result but since Hypothesis \hyperlink{hyp:H}{(H)} is not satisfied, the question is open.
	\end{example}

	\appendix\section{CY structures on graded categories}\label{A}

The following is a definition of $n$-CY structure on a $\mathbb Z$-graded $\bk$-linear category $\huaC$; it reduces to Definition \ref{defi:nCY} in the case when $\huaC$ is triangulated and $\huaC=\mathcal{L}=\mathcal{M}$.

\begin{defi}\label{CY-graded-cat}
Let $\huaC$ be a $\mathbb Z$-graded $\bk$-linear category and for all $x,y\in \huaC$, the graded vector space $\Hom^\bullet_{\huaC}(x,y)$ is finite-dimensional.  We call $\huaC$  is an $n$-CY category, if there exists degree $0$  bifunctorial isomorphisms of the $\mathbb Z$-graded $\bk$-linear vector spaces
\begin{eqnarray}\label{functorial iso}
\Phi_{x,y}:\Hom^\bullet_\huaC(x,y)\to \Hom^\bullet_\huaC(y,x)^*[-n],~\forall x,y\in\huaC.
\end{eqnarray}
\end{defi}

\begin{prop}
Suppose that the $\mathbb Z$-graded $\bk$-linear category $\huaC$ has an $n$-CY structure. Define degree $0$ pairings $\langle-,-\rangle:\Hom_\huaC^\bullet(x,y)\times \Hom^\bullet_\huaC(y,x)\to \bk[-n]$ as follows:
\begin{eqnarray}\label{graded-pairing-CY}
\langle f,g\rangle:=(-1)^{n|f|}\big(\Phi_{x,y}(f)[n]\big)(g),~~\forall f\in\Hom^\bullet_\huaC(x,y),g\in\Hom^\bullet_\huaC(y,x).
\end{eqnarray}
Then these pairings are non-degenerate and satisfy the following  condition of cyclic invariance:
\begin{eqnarray}\label{invariant condition}
\langle g\circ f,h\rangle=(-1)^{|g|(|f|+|h|)}\langle f\circ h,g\rangle.
\end{eqnarray}

Conversely, suppose that the $\mathbb Z$-graded $\bk$-linear category $\huaC$ has a collection of  degree $0$ non-degenerate pairings
$$\langle-,-\rangle:\Hom^\bullet_\huaC(x,y)\times \Hom^\bullet_\huaC(y,x)\to \bk[-n],
$$
which satisfy  \eqref{invariant condition}. Then there  exists a collection of degree $0$ bifunctorial isomorphisms
 $$
 \Phi_{x,y}:\Hom^\bullet_\huaC(x,y)\to \Hom^\bullet_\huaC(y,x)^*[-n],
 $$
 of the $\mathbb Z$-graded $\bk$-linear vector spaces, making $\huaC$ an $n$-CY category.

\end{prop}
\begin{proof}

 Since $\Phi_{x,y}$ is an isomorphism of $\bk$-graded vector spaces, it follows that the pairing $\langle-,-\rangle$ is non-degenerate. By \eqref{functorial iso}, for $f\in \Hom^\bullet_\huaC(x,y)$ and $g\in \Hom^\bullet_\huaC(y,z)$, we have  the following commutative diagram:
 \begin{equation}\label{functorial-1}
			\begin{split}
 \xymatrix{\Hom^\bullet_\huaC(x,y)  \ar[rr]^{\Phi_{x,y}}\ar[d]_{\Hom^\bullet_\huaC(x,g)} &  & \Hom^\bullet_\huaC(y,x)^*[-n] \ar[d]^{\Hom^\bullet_\huaC(g,x)^*[-n]}  \\
 \Hom^\bullet_\huaC(x,z)\ar[rr]^{\Phi_{x,z}} & &\Hom^\bullet_\huaC(z,x)^*[-n].}
\end{split}
		\end{equation}

 It follows that $\Phi_{x,z}\big(\Hom^\bullet_\huaC(x,g)(f)\big)=\big(\Hom^\bullet_\huaC(g,x)^*[-n]\big)\Phi_{x,y}(f)$. By \eqref{graded-pairing-CY}, for any $h\in \Hom_\huaC(z,x)$, we have
 \begin{eqnarray*}
\Big(\Phi_{x,z}\big(\Hom^\bullet_\huaC(x,g)(f)\big)[n]\Big)(h)=(-1)^{n(|f|+|g|)}\langle \Hom^\bullet_\huaC(x,g)(f),h\rangle=(-1)^{n(|f|+|g|)}\langle g\circ f,h\rangle.
 \end{eqnarray*}
 On the other hand, we deduce that
  \begin{eqnarray*}
\Big\{\Big(\big(\Hom^\bullet_\huaC(g,x)^*[-n]\big)\Phi_{x,y}(f)\Big)[n]\Big\}(h)&=&\Big(\Hom^\bullet_\huaC(g,x)^*\big(\Phi_{x,y}(f)[n]\big)\Big)(h)\\
                                                      &=&(-1)^{|g|(|f|-n)}\big(\Phi_{x,y}(f)[n]\big)(\Hom^\bullet_\huaC(g,x)(h))\\
                                                      &=&(-1)^{|g|(|f|-n)}(-1)^{|g||h|}\big(\Phi_{x,y}(f)[n]\big)(h\circ g)\\
                                                      &=&(-1)^{|g|(|f|+|h|-n)}(-1)^{n|f|}\langle f, h\circ g\rangle.
 \end{eqnarray*}
 Thus, we obtain
 \begin{eqnarray}\label{invariant condition-1}
 \langle g\circ f,h\rangle=(-1)^{|g|(|f|+|h|)}\langle f, h\circ g\rangle.
 \end{eqnarray}
 Similarly, by \eqref{functorial iso}, for $f\in \Hom^\bullet_\huaC(x,y)$ and $g\in \Hom^\bullet_\huaC(y,z)$, we have  the following commutative diagram:
 \begin{equation}\label{functorial-2}
			\begin{split}
 \xymatrix{\Hom^\bullet_\huaC(y,z)  \ar[rr]^{\Phi_{y,z}}\ar[d]_{\Hom^\bullet_\huaC(f,z)} &  & \Hom^\bullet_\huaC(z,y)^*[-n] \ar[d]^{\Hom^\bullet_\huaC(z,f)^*[-n]}  \\
 \Hom^\bullet_\huaC(x,z)\ar[rr]^{\Phi_{x,z}} & &\Hom^\bullet_\huaC(z,x)^*[-n].}
\end{split}
 \end{equation}

 Therefore, we have $\Phi_{x,z}\big(\Hom^\bullet_\huaC(f,z)(g)\big)=\big(\Hom^\bullet_\huaC(z,f)^*[-n]\big)\Phi_{y,z}(g)$. For any $h\in \Hom_\huaC(z,x)$,  we obtain
 \begin{eqnarray*}
\Big(\Phi_{x,z}\big(\Hom^\bullet_\huaC(f,z)(g)\big)[n]\Big)(h)=(-1)^{n(|f|+|g|)}\langle \Hom^\bullet_\huaC(f,z)(g),h\rangle=(-1)^{n(|f|+|g|)}(-1)^{|f||g|} \langle g\circ f,h\rangle.
 \end{eqnarray*}
 On the other hand, we deduce that
  \begin{eqnarray*}
\Big\{\Big(\big(\Hom^\bullet_\huaC(z,f)^*[-n]\big)\Phi_{y,z}(g)\Big)[n]\Big\}(h)&=&\Big(\Hom^\bullet_\huaC(z,f)^*\big(\Phi_{y,z}(g)[n]\big)\Big)(h)\\
                                                      &=&(-1)^{|f|(|g|-n)}\big(\Phi_{y,z}(g)[n]\big)(\Hom^\bullet_\huaC(z,f)(h))\\
                                                      &=&(-1)^{|f|(|g|-n)}\big(\Phi_{y,z}(g)[n]\big)(f\circ h)\\
                                                      &=&(-1)^{|f|(|g|-n)}(-1)^{n|g|}\langle g, f\circ h\rangle.
 \end{eqnarray*}
 Thus, the following equality holds:
  \begin{eqnarray}\label{invariant condition-2}
  \langle g\circ f,h\rangle=\langle g, f\circ h\rangle.
   \end{eqnarray}

  We set $z=x$ and $h=\id_x$, it follows that
 $$
 \langle g\circ f,\id_x\rangle\stackrel{\eqref{invariant condition-1}}{=}(-1)^{|g||f|}\langle f, g\rangle\stackrel{\eqref{invariant condition-2}}{=}\langle g, f\rangle.
 $$
 Therefore, the pairing is graded symmetric.
Moreover, we deduce that Eq.~\eqref{invariant condition-1} and Eq.~\eqref{invariant condition-2} are equivalent to the cyclic invariance condition
$$
\langle g\circ f,h\rangle=(-1)^{|g|(|f|+|h|)}\langle f\circ h,g\rangle.
$$

Conversely, suppose that the $\mathbb Z$-graded $\bk$-linear category $\huaC$ has a collection of degree $0$ non-degenerate pairings
$$\langle-,-\rangle:\Hom^\bullet_\huaC(x,y)\times \Hom^\bullet_\huaC(y,x)\to \bk[-n],
$$
which satisfy  condition \eqref{invariant condition}. One can define the degree $0$ morphisms of the $\mathbb Z$-graded $\bk$-linear vector spaces
\begin{eqnarray*}
\Phi_{x,y}:\Hom^\bullet_\huaC(x,y)\to \Hom^\bullet_\huaC(y,x)^*[-n],~\forall x,y\in\huaC,
\end{eqnarray*}
as follows:
\begin{eqnarray}\label{pairing-to-Iso}
\Phi_{x,y}(f):=(-1)^{n|f|}\langle f,-\rangle[-n],~~\forall f\in\Hom^\bullet_\huaC(x,y).
\end{eqnarray}
Since the pairings $\langle -,-\rangle$ are non-degenerate, we deduce that the morphisms $\Phi_{x,y},~\forall x,y\in\huaC$ are isomorphisms. Now we need to prove that the  morphisms $\Phi_{x,y}$ are bifunctorial, in other words, that  the diagrams \eqref{functorial-1} and \eqref{functorial-2} are commutative. By the definition of $\Phi_{x,y}$ \eqref{pairing-to-Iso}, the commutative diagram \eqref{functorial-1} is equivalent to the equality $\langle g\circ f,h\rangle=(-1)^{|g|(|f|+|h|)}\langle f, h\circ g\rangle$ and the commutative diagram \eqref{functorial-2} is equivalent to the equality $\langle g\circ f,h\rangle=\langle g, f\circ h\rangle$. Since the cyclic invariance condition \eqref{invariant condition} is equivalent to the equalities
\begin{eqnarray*}
\langle g\circ f,h\rangle&=&(-1)^{|g|(|f|+|h|)}\langle f, h\circ g\rangle,\\
\langle g\circ f,h\rangle&=&\langle g, f\circ h\rangle,
\end{eqnarray*}
this proves that the morphisms $\Phi_{x,y},~\forall x,y\in\huaC$ are bifunctorial isomorphisms. Therefore, $\huaC$ is an $n$-CY-category as claimed.
\end{proof}

\begin{rem}
We have proved that a collection of pairings $\langle-,-\rangle:\Hom^\bullet_\huaC(x,y)\times \Hom^\bullet_\huaC(y,x)\to \bk[-n]$ satisfying \eqref{invariant condition} gives an equivalent definition of an $n$-CY structure. This is, in turn, equivalent to the  condition $\langle g\circ f,h\rangle=\langle g,f\circ h\rangle$ plus the graded symmetry of the pairing $\langle,\rangle$, which in the triangulated case specializes to Keller's definition, \cite[Section 1.1]{Kel}. Definition \ref{CY-graded-cat} has an obvious analogue for dg categories; however in the latter case it is customary to impose a further condition involving cyclic homology of $\mathcal C$, see \cite[Definition 3.2]{DB}.
\end{rem}

\begin{rem}
	Let $\mathcal C$ have a single object, i.e. $\C$ is simply a graded $k$-algebra. Then Proposition \eqref{graded-pairing-CY} reduces to the well-known
	statement that a graded algebra $\mathcal C$ possesses an invariant graded symmetric inner product if and only if there exists an isomorphism $\mathcal C\to \mathcal C^*$ of $\mathcal C$-bimodules.
	\end {rem}

\noindent
{\bf Acknowledgement.} We would like to express our thanks to anonymous referees for a host of useful suggestions and corrections, especially pointing out to us issues stemming from the quasi-isomorphism non-invariance of the function $V\mapsto|V|$ for a dg vector space $V$. The first author (A. Lazarev) is grateful to M. Booth, J. Chuang  and T. Voronov for useful discussions on Calabi-Yau categories and duality. The second author (R. Tang) was partially supported by NSFC (W2412041, 12371029).

\noindent
{\bf Data Availability.} No datasets were generated or analyzed
	as part of the present research.

\noindent
{\bf Declarations.}
The authors have no conflict of interest to declare that is relevant to the content of this
	paper.

\end{document}